\documentclass[11pt,a4paper,reqno]{amsart}
\usepackage{calc,graphicx,amsfonts,amsthm,amscd,epsfig,psfrag,amsmath,amssymb,enumerate,dsfont}

\usepackage{calc,graphicx,amsfonts,amsthm,amscd,epsfig,psfrag,amsmath,amssymb,enumerate} 

\usepackage[showlabels,sections,floats,textmath,displaymath]{}
\usepackage[colorlinks=true, pdfstartview=FitV, linkcolor=blue, citecolor=blue, urlcolor=blue,pagebackref=false]{hyperref}
\numberwithin{equation}{section} 

\DeclareMathSymbol{\leqslant}{\mathalpha}{AMSa}{"36} 
\DeclareMathSymbol{\geqslant}{\mathalpha}{AMSa}{"3E} 
\DeclareMathSymbol{\eset}{\mathalpha}{AMSb}{"3F}     
\renewcommand{\leq}{\;\leqslant\;}                   
\renewcommand{\geq}{\;\geqslant\;}                   

\def\1{\ifmmode {1\hskip -3pt \rm{I}} \else {\hbox {$1\hskip -3pt \rm{I}$}}\fi}

\newtheorem{Theorem}{Theorem}[section]
\newtheorem{Lemma}[Theorem]{Lemma}
\newtheorem{Proposition}[Theorem]{Proposition}
\newtheorem{Corollary}[Theorem]{Corollary}
\newtheorem{remark}{Remark}[section]

\newtheorem{claim}[Theorem]{Claim}
\newtheorem{definition}[Theorem]{Definition}

\newcommand{\cA}{\ensuremath{\mathcal A}}

\newcommand{\cD}{\ensuremath{\mathcal D}}

\newcommand{\cF}{\ensuremath{\mathcal F}}
\newcommand{\cG}{\ensuremath{\mathcal G}}

\newcommand{\cL}{\ensuremath{\mathcal L}}

\newcommand{\cZ}{\ensuremath{\mathcal Z}}

\newcommand{\bbA}{{\ensuremath{\mathbb A}} }

\newcommand{\bbE}{{\ensuremath{\mathbb E}} }
\newcommand{\bbF}{{\ensuremath{\mathbb F}} }

\newcommand{\bbI}{{\ensuremath{\mathbb I}} }

\newcommand{\bbL}{{\ensuremath{\mathbb L}} }

\newcommand{\bbN}{{\ensuremath{\mathbb N}} }

\newcommand{\bbP}{{\ensuremath{\mathbb P}} }

\newcommand{\bbR}{{\ensuremath{\mathbb R}} }

\newcommand{\bbZ}{{\ensuremath{\mathbb Z}} }

\newcommand{\ent}{{\rm Ent} }
\newcommand{\var}{\operatorname{Var}}

\newcommand{\bra}{\langle}
\newcommand{\ket}{\rangle}

\let\a=\alpha \let\b=\beta   \let\d=\delta  \let\e=\varepsilon
 \let\g=\gamma \let\h=\eta      \let\l=\lambda
          \let\p=\pi  
  \let\s=\sigma \let\t=\tau   
  
     \let\L=\Lambda 
\let\O=\Omega      

\def\\{\hfill\break}

\def\thsp{\thinspace}

\def\tthsp{\kern .083333 em}

\def\?{\mskip -10mu}
\def\indbox#1{\hbox to \parindent{\hfil\ #1\hfil} }

\newfam\msafam
\newfam\msbfam
\newfam\eufmfam
\def\hexnumber#1{%
  \ifcase#1 0\or 1\or 2\or 3\or 4\or 5\or 6\or 7\or 8\or
  9\or A\or B\or C\or D\or E\or F\fi}
\font\tenmsa=msam10 \font\sevenmsa=msam7 \font\fivemsa=msam5
\textfont\msafam=\tenmsa \scriptfont\msafam=\sevenmsa
\scriptscriptfont\msafam=\fivemsa
\edef\msafamhexnumber{\hexnumber\msafam}%
\mathchardef\restriction"1\msafamhexnumber16 \mathchardef\ssim"0218
\mathchardef\square"0\msafamhexnumber03
\mathchardef\eqd"3\msafamhexnumber2C
\def\QED{\ifhmode\unskip\nobreak\fi\quad
  \ifmmode\square\else$\square$\fi}
\font\tenmsb=msbm10 \font\sevenmsb=msbm7 \font\fivemsb=msbm5
\textfont\msbfam=\tenmsb \scriptfont\msbfam=\sevenmsb
\scriptscriptfont\msbfam=\fivemsb

\font\teneufm=eufm10 \font\seveneufm=eufm7 \font\fiveeufm=eufm5
\textfont\eufmfam=\teneufm \scriptfont\eufmfam=\seveneufm
\scriptscriptfont\eufmfam=\fiveeufm

\def\({\left(}
\def\){\right)}
\let\neper=e
\let\ii=i

\def\ie{\hbox{\it i.e.\ }}

\let\emp=\emptyset
\let\sset=\subset

\let\setm=\backslash
\def\nep#1{ \neper^{#1}}

\let\imp=\Rightarrow

\def\tc{\thsp | \thsp}
\let\<=\langle
\let\>=\rangle

\def\Var{ \mathop{\rm Var}\nolimits }

\def\gap{\mathop{\rm gap}\nolimits}

\def\scalprod#1#2{ \thsp<#1, \thsp #2>\thsp }
\def\inte#1{\lfloor #1 \rfloor}

\outer\def\nproclaim#1 [#2]#3. #4\par{\medbreak \noindent
   \talato(#2){\bf #1 \Thm[#2]#3.\enspace }%
   {\sl #4\par }\ifdim \lastskip <\medskipamount
   \removelastskip \penalty 55\medskip \fi}
\def\thmm[#1]{#1}
\def\teo[#1]{#1}
\def\sttilde#1{%
\dimen2=\fontdimen5\textfont0 \setbox0=\hbox{$\mathchar"7E$}
\setbox1=\hbox{$\scriptstyle #1$} \dimen0=\wd0 \dimen1=\wd1
\advance\dimen1 by -\dimen0 \divide\dimen1 by 2
\vbox{\offinterlineskip%
   \moveright\dimen1 \box0 \kern - \dimen2\box1}
}
\begin{document}
\title[East Model]{The East model: recent results and new progresses}

\author{A. Faggionato}
\address{Alessandra Faggionato. Dip. Matematica ``G. Castelnuovo", Univ. ``La
  Sapienza''. P.le Aldo Moro  2, 00185  Roma, Italy. e--mail:
  faggiona@mat.uniroma1.it}

 \author[F. Martinelli]{F. Martinelli}
 \address{F. Martinelli. Dip. Matematica, Univ. Roma Tre, Largo S.L.Murialdo 00146, Roma, Italy. e--mail:
martin@mat.uniroma3.it
 }

 \author[C. Roberto]{C. Roberto}
 \address{Cyril Roberto. Modal'X, Univ. Paris Ouest Nanterre, 200 av R\'epublique 92000 Nanterre, France. e--mail:
cyril.roberto@math.cnrs.fr }

\author[C. Toninelli]{C. Toninelli}
\address{Cristina Toninelli. L.P.M.A. and
  CNRS-UMR 7599, Univ. Paris VI-VII 4, Pl. Jussieu 75252
  Paris, France. e--mail: cristina.toninelli@upmc.fr}

\thanks{\sl Work supported by the European Research Council through the ``Advanced
Grant'' PTRELSS 228032.}



\begin{abstract}
The East model is a particular one dimensional interacting particle
system in which certain transitions are forbidden according to some
constraints depending on the configuration of the system. As such it
has received particular attention in the physics literature as a
special case of a more general class of systems referred to as
\emph{kinetically
  constrained models}, which
play a key role in explaining some features of the dynamics of
glasses. In this paper we give an extensive overview of  recent rigorous results concerning
the equilibrium and non-equilibrium dynamics of the East
model together with some new improvements.
 \end{abstract}

\maketitle

\thispagestyle{empty}

\section{Introduction}
Facilitated or kinetically constrained spin (particle) models (KCSM) are
interacting particle systems which have been introduced in the physics
literature \cite{FA1,FA2,JACKLE} to model liquid/glass transition and more
generally ``glassy dynamics'' (see  e.g. \cite{Ritort,GarrahanSollichToninelli}). A configuration is given by
assigning to each vertex $x$ of a (finite or infinite) connected graph
$\cG$ its occupation variable
$\eta(x)\in\{0,1\}$ which corresponds to an empty or filled site,
respectively.  The evolution is given by a Markovian stochastic dynamics
of Glauber type.  Each site with rate one refreshes its occupation variable to
a filled or to an empty state with probability $1-q$ or $q$
respectively provided that the current configuration around it satisfies an
a priori specified constraint.  For each site $x$ the corresponding constraint does not
involve $\h(x)$, thus detailed balance w.r.t. the Bernoulli($1-q$) product
measure $\pi$ can be easily verified and the latter is an invariant
reversible measure for the process.

One of the most studied KCSM is the East model
\cite{JACKLE}\footnote{\ Quite interestingly the East model plays a key
  role in certain random walks over the upper triangular matrices with
  entries in the field $\bbZ_n$ for $n$ prime \cite{Peres-Sly}.}.
It  is  a one-dimensional  model ($\cG=\bbZ$ or
$\cG=\bbZ_+=\{0,1,\dots\}$) and particle creation/annihilation at a
given site $x$ can occur only if the \emph{East neighbor} of $x$, namely
the vertex $x+1$, is empty. The model is ergodic for any $q \in (0,1)$
with a positive spectral gap \cite{Aldous,CMRT} and it relaxes to
the equilibrium reversible measure exponentially fast even when
started from e.g.\ any non-trivial product measure \cite{CMST}.
However, due to the fact that the rates can be zero, the East model
has specific features quite different from those of more common
systems. For example the relaxation time $T_{\rm relax}(q)$ diverges
very fast as $q\downarrow 0$, $T_{\rm
relax}\sim\left(1/q\right)^{\frac 12 \log_2(1/q)}$ (see
\cite{CMRT}), and several coercive inequalities stronger than the
Poincar\'e inequality (e.g.\  the logarithmic Sobolev inequality)
fail (see Section \ref{logsob} for more details).

A key issue, both from the mathematical and the physical point of view, is
therefore that of describing accurately the evolution at $q\ll 1$ when the initial distribution is different from the reversible one
and for time scales which are large but still much smaller
than $T_{\rm relax}(q)$ when the exponential relaxation to the
reversible measure takes over. A typical case, often referred
to in then physics literature as a quench from high to low
density of vacancies, is to take as starting distribution i.i.d.\ occupancy
variables with density $1/2$.
We refer the interested reader to
\cite{Ritort,Leonard,Crisanti,GarrahanNewman,CorberiCugliandolo} for
the relevance of this setting in connection with the study of the
liquid/glass transition as well as for details for KCMS different from East model.

As first suggested in the non-rigorous work \cite{SE1,SE2} and
recently mathematically established in \cite{FMRT-cmp}, the
non-equilibrium dynamics of the East model as $q\downarrow 0$ is
dominated by a metastable type of evolution in a energy landscape
with a hierarchical structure. Such metastable dynamics can in turn
be very well described by a hierarchical coarsening process
\cite{FMRT-cmp} for the  excess vacancies whose long time behavior can
be analyzed rigorously. Remarkably such a hierarchical coalescence
process (i) has exactly the same general structure of other
coalescence processes introduced in the physics literature for very
different situations (see e.g.\ \cite{D0,D1,D2}) and (ii) the form of
its universality classes can be mathematically established and
computed \cite{FMRT}. As a consequence one is able to draw almost exact
conclusions on the out-of-equilibrium dynamics of the East model \cite{FMRT-cmp}.


In this paper we mostly try to provide an extensive self-contained
review of the existing mathematical theory of the East model in the
various regimes. We also provide  the analysis of the   logarithmic
Sobolev inequaliy and some of its recently introduced modifications,
as well as some extension of the main theorems proved in
\cite{FMRT-cmp} for the low density non-equilibrium dynamics.

Finally, we stress that some of the results and/or the techniques that we present  are valid for more general KCSM (e.g.\ spectral gap, persistence, log-Sobolev); while others are related to the oriented and/or one dimensional character of the East Model.

\setcounter{tocdepth}{3}
\tableofcontents

\section{The East process: definition and construction}
\subsection{Notation}

Throughout all the paper we will use the notation
$\bbN:=\{1,2,\dots\}$ and $\bbZ_+:=\{0,1,2,\dots\}$. The
configuration space  for the East model is either
$\Omega:=\{0,1\}^{\bbZ}$ or $\Omega_\Lambda = \{0,1\}^\Lambda$ for
some (finite or infinite) subset $\Lambda \subset \bbZ$. Given a
parameter $q \in[0,1]$, for any $x \in \bbZ$, $\pi_x$  denotes a
Bernoulli$(1-q)$ measure, $\pi:=\prod_{x\in\bbZ}\pi_x$ and
$\pi_\Lambda:=\prod_{x\in\Lambda}\pi_x$ for $\Lambda \subset \bbZ$.
Also, we set  $p:=1-q$.

Elements of $\O$ will usually be denoted by the Greek letters
$\sigma,\eta,\dots$ and $\s(x)$ will denote the occupancy variable
at the site $x$: when $\sigma(x)=1$ we say that site $x$ is occupied
or filled (by a particle), while when $\sigma(x)=0$ we say  that
there is a vacancy (or no particle) at site $x$, or also that $x$ is
empty. The restriction of a configuration $\s$ to a subset $\L$ of
$\bbZ$ will be denoted by $\s_\L$.
 Given two sets $\Lambda, V$ and two configurations $\sigma$, $\sigma'$, $\sigma_\Lambda \sigma'_V$ denotes the configuration equal to $\sigma$
 on $\Lambda$ and to $\sigma'$ on $V$. The set of empty sites (or zeros) of a configuration $\s$ will be
denoted by $\cZ(\s)$.

 The mean with respect to $\p$  of a function $f $ on $\O$  
 is denoted by $\pi(f)$,
  while its  variance
is denoted by $\var(f)$. Similar definitions hold for
$\pi_\Lambda(f), \var_\Lambda (f) $ and $f$ a function on $\O_\L$. If
$f$ is a function on  $ \O$  we denote by $\pi_\Lambda(f)$ and $
\var_\Lambda (f) $ the mean and the variance of $f$ with respect to
the conditional probability $\p\bigl( \cdot | \{\s(y) \}_{y \in
\L}\bigr)= \p\bigl( \cdot | \s_\L\bigr)$.
 Namely, $\p_\Lambda (f)$ is the
mean of $f$ with respect to $\p_\L$ computed keeping fixed the
variables $\s(y)$, $y \not \in \L$. Similarly for $\var _\Lambda
(f)$.
 For simplicity, we set $\p_x(f):=
\p_{\{x\}}(f)$ and $\var_x(f):= \var_{\{x\}}(f)$.

Finally we introduce the entropy functional $\ent(f):=\pi(f \log
(f/\pi(f))$ for any non-negative function $f$, say in
$\bbL^2(\Omega,\pi)$, and similarly $\ent_\Lambda(f) = \pi_\Lambda(f
\log (f/\pi_\Lambda (f)) $.

\begin{remark}\label{temp}
In the physical literature, the parameter
$q$, which represents the density of vacancies as will become clear later, is written as $q=\frac{\nep{-\b}}{1+\nep{-\b}}$ where $\b$ is the inverse temperature. In particular, the limit $q\downarrow 0$ corresponds to the zero temperature limit.
\end{remark}

\subsection{Infinitesimal  generator of the East process}
The East  process can be informally described as follows. Each
vertex $x$ waits an independent mean one exponential time and then,
provided that the current configuration $\sigma$ satisfies the
constraint $\sigma(x+1)=0$, the value of $\sigma(x)$ is refreshed
according to $\pi_x$, {\it i.e.}\ it
 is set equal to $1$ with
probability $p=1-q$ and to $0$ with probability $q$. The process can
be rigorously constructed in a standard way, see \cite{Liggett1}.
Formally, it is univocally specified by the action of its
infinitesimal Markov generator $\cL$ on local  (\ie depending on
finitely many variables) functions $f \colon \O \mapsto \bbR$, given
by
\begin{align}
\label{thegenerator}
\cL f(\sigma)
& =
\sum_{x\in \bbZ}c_{x}(\sigma)\left[\pi_x(f)-f(\sigma)\right] \\
& =
\sum_{x\in \bbZ}c_{x} (\sigma)\left[(1-\sigma(x))p+\sigma(x)q\right]\left( f(\sigma^x) - f(\sigma) \right) \nonumber
\end{align}
where $c_x(\sigma):=1-\sigma(x+1)$ encodes the constraint,
and $\sigma^x$ is obtained from $\sigma$ by flipping its value
at $x$, {\it i.e.}
$$
\sigma^x(y)=\left\{
\begin{array} {ll}
\sigma(y) & \mbox{if } y \neq x \\
1-\sigma(x) & \mbox{if } y=x
\end{array}
\right. .
$$
The domain of $\cL$ is denoted by $\mathrm{Dom}(\cL)$. When the
initial distribution at time $t=0$ is $Q$, the law and expectation
of the process on the Skohorod space $D([0,\infty) , \O)$ will be
denoted by $\bbP_Q$ and $\bbE_Q$ respectively. If $Q=\d_{\s}$ we
write simply $\bbP_\s$. The process at time $t$ will be denoted by
$\sigma_t$.

The East process can  also be defined on intervals $\L=[a,b]\subset
\bbZ$  or half-lines $\L=(-\infty, b]\subset \bbZ$ provided that a
suitable \emph{zero} boundary condition is specified at the site
$b+1$. More precisely one defines the generator $\cL_\Lambda$
(called finite volume generator if $\Lambda$ is  finite), acting on
any local function $f\colon \Omega_\L \to \mathbb{R}$ as
\begin{gather}
\cL_{\L} f(\s) = \sum_{x\in \L}c_{x}(\s)\left[\mu_x(f)-f(\s)\right]+
\left[\mu_b(f)-f(\s)\right]
    \equiv \sum_{x\in \L}c_{x}^{\L}(\s)\left[\mu_x(f)-f(\s)\right]\, ,\nonumber\\
   \text{where}\quad c_x^{\L}(\s)=\begin{cases}
    1-\s(x+1) &\text{for $x\in\L \setminus \{b\}$}\\1 &\text{if $x=b$}
   \end{cases}\,.
\label{bersani}
\end{gather}
In particular there is no constraint at site $b$ as a \emph{frozen
zero} lies at site $b+1$. One can define the East process also on
half-lines $\L=[b, \infty)$, in this case the generator is given by
$$\cL_\L f (\s)= \sum _{x\in [b,
\infty)}c_{x}(\s)\left[\mu_x(f)-f(\s)\right]\,.$$


Due to the ``East''
  character of the constraint, taking $\L$ as above (\ie $\L=[a,b]$, $(-\infty,b]$, $[b, \infty)$), for any initial condition $\h\in \O_\L$
  the
  evolution  on the   interval $\L$ coincides with that of the East process on $\bbZ$ (restricted to $\L$)
  starting from the configuration
    \begin{equation}\label{estendiamo}
    \tilde \h (x):=
    \begin{cases} \h(x) & \text{ if } x\in \L\,,\\
    0 & \text{ if }x=b+1\,, \; b = \max \L\,,\\
     1 & \text{ otherwise}\,.
     \end{cases}
     \end{equation}
We will use the self-explanatory notation $\bbP_Q ^{\L}$ (or
$\bbP_\s^{\L}$) for the law of the process starting from the law $Q$ (from $\s$)
and $\sigma_t^\Lambda$ for the process at time $t$.

\medskip

Note that, by construction, the East process on $\bbZ$ (respectively
on $\Lambda$) is reversible with respect to $\pi$ (respectively
$\pi_\Lambda$). Analytically this is equivalent to say that $\cL$
(respectively $\cL_\Lambda$) is a self-adjoint operator in
$\bbL^2(\pi)$ (respectively $\bbL^2(\pi_\Lambda)$). Moreover, for
any   $f, g \in \mathrm{Dom}(\cL)$ (in particular, for any local
functions $f,g$), the Dirichlet form associated to the generator
$\cL$ is given by
$$
\cD(f,g):=\frac{1}{2} \sum_{x\in \bbZ} \sum_{\sigma \in \Omega}
\pi(\sigma) c_{x} (\sigma)\left[(1-\sigma(x))p+\sigma(x)q\right]
\left( f(\sigma^x) - f(\sigma) \right)\left( g(\sigma^x) - g(\sigma)
\right) \,.
$$
Similarly the Dirichlet form associated to the generator $\cL_\L$ is
given by
$$
\cD_\Lambda(f,g):=\frac{1}{2} \sum_{x\in \Lambda} \sum_{\sigma \in
\Omega_\L} \pi_\Lambda(\sigma) c_{x}^\Lambda
(\sigma)\left[(1-\sigma(x))p+\sigma(x)q\right] \left( f(\sigma^x) -
f(\sigma) \right)\left( g(\sigma^x) - g(\sigma) \right) \!.
$$
It is simple to check that
\begin{align}
& \cD(f):=\cD(f,f)= \sum_{x \in \mathbb{Z}} \pi \left( c_x \var_x
(f)
\right)\qquad f \in \mathrm{Dom} ( \cL )\,, \\
& \cD_\Lambda(f):=\cD_\Lambda(f,f)=\sum_{x \in \Lambda} \pi_\Lambda
\left( c_x^\Lambda \var_x (f) \right)) \qquad f \in \mathrm{Dom}(
\cL_\L) \,. \label{freddo2}
\end{align}
In what follows, when considering a local function $f $ on $\O$ we
denote by $\cD_\Lambda (f)$ the Dirichlet form of $f$ with respect
to $\cL_\L$ and  $\p_\L$ computed holding fixed the variables
$\{\s(y)\}_{y \not \in \L}$. In particular, \eqref{freddo2} still
holds due to our definition of $\p_\L(\cdot)$ and $\var_x (\cdot)$
for functions defined on $\O$.


Finally we introduce the associated Markov semigroup $P_t=e^{t \cL}$
which satisfies $P_t f(\sigma)=\mathbb{E}_\sigma(f(\sigma_t))$ and
similarly $P_t^\Lambda=e^{t \cL_\Lambda}$.


\subsection{Graphical construction} \label{graphical}
In this section we briefly recall a standard graphical construction
which allows to define on the same probability space the finite
volume East process for \emph{all} initial conditions. Using a
standard percolation argument, see \cite{Durrett,Liggett2}, together
with the fact that the constraints $c_x$ are uniformly bounded and
of finite range, it is not difficult to see that the graphical
construction can be extended  to any infinite volume.

Given a finite interval $\L\sset \bbZ$ we associate to each $x\in \L$  a
Poisson process  of parameter one  and, independently, a family of
independent Bernoulli$(1-q)$  random variables $\{s_{x,k}:k\in
\bbN\}$ (coin tosses). The occurrences of the Poisson process associated to $x$ will
be denoted by $\{t_{x,k}:\ k\in \bbN\}$. We assume independence as $x$
varies in $\L$. Notice that with probability one all the
occurrences $\{t_{x,k}\}_{k\in \bbN,\, x\in \Lambda}$ are
different. This defines the probability space. The corresponding
probability measure will be denoted by $\bbP$. Given an initial
configuration $\h\in \O$ we construct a Markov process
$(\s_t^{\L,\h})_{t\ge 0}$ on the above probability space satisfying
$\s^{\L,\h}_{t=0}=\h$  according to the
following rules.
At each time $t=t_{x,n}$ the site $x$ queries the state of its own
constraint $c^\L_x$. If the constraint  is satisfied, \ie if
$\s^{\L,\h}_{t-}(x+1)=0$, then $t_{x,n}$ will be  called a \emph{legal ring} and
at time $t$ the configuration resets its value at site $x$  to the value of the corresponding Bernoulli variable
$s_{x,n}$. We stress here that
the rings and coin tosses at $x$ for $s\le t$ have no influence whatsoever on the
evolution of the configuration at the sites which enter in its constraint (here
$x + 1$) and thus they have no influence of whether a ring at $x$ for $s > t$ is legal
or not. It is classical to see that the above construction actually gives a
continuous time Markov chain with generator \eqref{bersani}.

A simple but important consequence of the graphical
construction is the following one. Assume
that the zeros of the starting configuration $\s$ inside $\Lambda$ are labeled in
increasing order as $x_0,x_1,\dots,x_n$ and define $\t$ as the first time
at which one the $x_i$'s is killed, \ie the occupation variable there flips to
one. Then, up to time $\t$ the East dynamics factorizes over the East
process in each interval $[x_i, x_{i+1})$.

\section{Main results}

In this section we collect the most relevant rigorous results on the
East model. The main references are
\cite{Aldous,CMRT,CMST,FMRT,FMRT-cmp,bodineau-toninelli}, while two
results are new: the analysis of the $\a$--log--Sobolev inequalities
(see Section \ref{logsob})  and an extension of the results of
\cite{FMRT-cmp} (see Section \ref{sec:outIIintro}).

\subsection{Spectral gap}
The finite volume East process is trivially ergodic because of the
frozen zero boundary condition (see \eqref{bersani}). The infinite
volume process in $\bbZ$ is also ergodic in the sense that $0$ is a
simple eigenvalue of the generator $\cL$, as proved in \cite{CMRT}.
This last property implies that the process converges to
equilibrium, in $\bbL^2(\pi)$. More precisely (see e.g.\
\cite[Theorem 4.13, Ch IV]{Liggett1}), the following classical equivalence
holds.
\begin{Theorem}
\label{ergodic}
  The following properties are equivalent:
  \begin{enumerate}[(a)]
  \item $\lim_{t\to \infty} \bbE_\sigma(f(\sigma_t))=\pi(f)$ in $\bbL^2(\pi)$ for all
    $f\in \bbL^2(\pi)$;
\item $0$ is a simple eigenvalue for $\cL$.
  \end{enumerate}
\end{Theorem}

 The next natural question that arises  is how fast the process converges to equilibrium. The classical tool to
answer such a question is a spectral gap estimate.
We recall that  the spectral gap (or inverse of the relaxation time)
of the generator $\cL$ is defined as
\begin{equation}\label{def_gap}
\gap(\cL):= \inf_{\genfrac{}{}{0pt}{}{f \in \mathrm{Dom}(\cL)}{f
\neq \mathrm{const}}} \frac{\cD(f)}{\var(f)}=
\inf_{\genfrac{}{}{0pt}{}{f \mathrm{ local } }{f \neq
\mathrm{const}}} \frac{\cD(f)}{\var(f)}\,.
\end{equation}
Similarly  one defines the spectral gap, $\gap(\cL_\Lambda)$, of the
generator $\cL_\Lambda$  for $\Lambda \subset \bbZ$. It is
well-known (see {\it e.g.}\ \cite[Chapter 2]{ane}) that $\gap(\cL) \geq \gamma$ for some $\gamma>0$
is equivalent to the following exponential decay of the semigroup:
$$
\var(P_t f) = \int [\bbE_\sigma(f(\sigma_t)) - \pi(f)]^2
d\pi(\sigma) \leq \var(f) e^{-2\gamma t} \qquad \forall t >0, \quad
\forall f    \in \mathrm{Dom}(\cL) .
$$
The next result asserts that $\gap (\cL) >0$ so that the process indeed converges to equilibrium exponentially fast, in $\bbL^2(\pi)$.
Moreover one can compute the precise asymptotic of $\gap (\cL)$ in the limit $q \downarrow 0$.

\begin{Theorem}[\cite{Aldous,CMRT,CMST}] \label{th:gap}
The following holds:
\begin{itemize}
\item[(i)]
The generator $\cL$ has a positive spectral gap, \ie $\gap(\cL) >0$;
\item[(ii)] The asymptotic of $\gap(\cL)$ for $q \downarrow 0$ is
given by
\begin{equation*}
\lim_{q\downarrow
0}\log(\gap(\cL)^{-1})/\left(\log(1/q)\right)^2=(2\log 2)^{-1}\,;
\end{equation*}
\item[(iii)]
For any interval $\L \subset \bbZ$, the spectral gap of the finite
volume generator $\cL_\L$ is not smaller than $\gap(\cL)$, \i.e.
$\gap(\cL_\Lambda) \geq \gap(\cL)\,$.
\end{itemize}
\end{Theorem}
\begin{remark}
Points  $(i)$ and $(iii)$ have been proven for the first time by
Aldous and Diaconis in \cite{Aldous}. These authors also showed the
correct upper bound in $(ii)$ together with a lower bound that is
off by a factor $1/2$. This wrong factor $1/2$ also appeared in the
conjectured behavior of the relaxation time suggested in the physics
literature \cite{SE1,SE2} and based on simple energy barriers
considerations. The discrepancy with the correct   asymptotic as
given in (ii) is mainly due to neglecting an important contribution
coming from the entropy (\ie number of ways to overcome the energy
barrier). The matching lower bound was proven in \cite{CMRT} by a
completely novel approach while an alternative (and somehow very
natural) proof of the upper bound can also be found in \cite{CMST}.
The necessary techniques will be developed in Section
\ref{sec:coercive} where the reader will find the complete proof of
Theorem \ref{th:gap}.

The techniques developed in \cite{CMRT} to prove the positivity of the spectral gap (Item $(i)$)
are actually valid for a wide class of KCSM (not necessarily one dimensional).
\end{remark}

\subsection{Persistence function}
We now consider  the persistence function $F(t)$ which represents
the probability for the equilibrium process that the occupation
variable at the origin does not change before time $t$. More
precisely (see e.g.\ \cite{Ha,SE2}) the persistence function is
defined by
\begin{equation} \label{eq:Pers}
F(t):=\int d\pi(\h)\; \bbP_\h (\s_0(s)=\h_0,\; \forall s\le t) .
\end{equation}

In \cite{CMRT}, using a Feynman-Kac formula approach, it is  proved that
$F(t)$ decays exponentially fast as predicted in the physics literature.
\begin{Theorem}(\cite{CMRT}) \label{th:persistence}
It holds
$$
F(t) \leq 2 \exp\left\{-\frac{\gap(\cL) \min(p,q)}{4} t \right\} \qquad \qquad \forall t \geq 0 .
$$
\end{Theorem}

The proof of the latter is given in Section \ref{sec:persistence}.
As for the positivity of the spectral gap, note that the exponential decay of the persistence function holds for more general KCSM
\cite{CMRT}.

\subsection{Log-Sobolev constant}
\label{logsob}

The next step in understanding the long-time behavior of the East model is the study of the log-Sobolev constant.
This coercive constant is usually used to prove exponential decay in
the  sup-norm (and therefore in a stronger sense compared to $\mathbb{L}^2(\pi)$),
see \cite{holley-stroock,SFlour}), by means of the celebrated
hypercontractivity property. Unfortunately, for the East Model the log-Sobolev constant in infinite volume does not exist.

In fact, a whole family of Sobolev type inequalities does not hold. This family is called $\alpha$-log-Sobolev inequalities,
$\alpha \in [0,2]$ being a parameter. In  a finite interval $\Lambda$ they are defined as follows.

Given $\alpha \in (0,2]\setminus \{1\}$, one says that $\pi_\Lambda$ satisfies the $\alpha$-log-Sobolev inequality if
there exists some constant $C_\alpha(\Lambda) \in (0,\infty)$ such that,
for any $f \colon \Omega_\Lambda \to \mathbb{R}$, it holds
\begin{equation} \label{ls}
\ent_{\Lambda}(f) \leq \frac{\alpha \alpha' C_\alpha(\Lambda)}{4} \mathcal{D}_\Lambda (f^{1/\alpha},f^{1/\alpha'})
\end{equation}
where $\alpha'$ is the dual exponent of $\alpha$, {\it i.e.}\ such that $\frac{1}{\alpha} + \frac{1}{\alpha'} =1$.
Observe that, since $\alpha$ may belong to $(0,1)$, $\alpha'$ may be negative. However, due to the multiplicative factor
$\alpha'$, the right hand side is always non-negative. The
$\a$-log-Sobolev inequalities with $\a=1$ or $\a=0$ are defined
by a limiting procedure.

Such a family has been introduced in \cite{mossel} as an interpolating family from the log-Sobolev inequality
to the Poincar\'e inequality.
Indeed, for $\alpha=2$ inequality \eqref{ls} reduces to the standard log-Sobolev inequality of Gross \cite{gross}. Also, in
\cite[Section 4]{mossel}, it is proved that the $0$-log-Sobolev inequality, with constant $C_0$,
is precisely equivalent to the standard Poincar\'e inequality with constant $C_0/2$.
Moreover, the limiting case $\alpha=1$ is equivalent to the following inequality independently studied in the literature:
\begin{equation} \label{lsm}
\ent_{\Lambda}(f) \leq \frac{C_1(\Lambda)}{4} \mathcal{D}_\Lambda (f, \log f) .
\end{equation}
The latter\footnote{Inequality \eqref{lsm} is sometimes called "modified logarithmic Sobolev inequality"
\cite{bobkov-ledoux,gao-quastel,goel,bobkov-tetali} or "entropy
inequality" \cite{daipra,caputo-posta,caputo}. Recently a yet new
name, \emph{$1$-log-Sobolev
inequality}, has been introduced always for the same object.}
has been introduced in \cite{bobkov-ledoux} to study the concentration phenomenon of birth and death processes on the integer line. It is known \cite{diaconis-saloff-coste} that \eqref{lsm} is actually equivalent to the following exponential decay
to equilibrium of the semi-group, in the entropy sense:
$$
\ent_{\Lambda}(P_t f) \leq e^{-4t/C_1(\Lambda)}  \ent_{\Lambda}(f) \qquad \qquad \forall t \geq 0 .
$$
Hence a control on the constants $C_\alpha(\Lambda)$ may reveal to be crucial in the study of the long-time behavior of the dynamics,
specially for $\alpha=1,2$.

In the next theorem, we prove that, for any $\alpha \in (0,2]$ the
constant $C_\alpha(\Lambda)=\O(|\L|)$ compares to $|\Lambda|$. Thus,
in order to get exponential decay to equilibrium either in the
sup-norm or in the entropy sense, one has to use alternative
strategies. One of them will be developed in the next section. Also,
Theorem \ref{th:ls} below answers partially to a question asked to
us by Krzysztof Oleszkiewicz (see \cite[Section 12]{mossel}), namely
about the existence of an example for which the Poincar\'e
Inequality holds while none of the $\alpha$-log-Sobolev
inequalities, $\alpha \in (0,2]$, hold.

\begin{Theorem} \label{th:ls}
Fix $\alpha \in (0,2]$ and a finite interval $\Lambda$ of $\bbZ$. Let $C_\alpha(\Lambda)$ be the best possible ({\it i.e.}\ the smallest) constant in Inequality \eqref{ls}.
Then, there exists a constant $c$ (that may depend on $q$ and $\alpha$) such that
$$
\frac{1}{c} |\Lambda| \leq C_\alpha(\Lambda) \leq c |\Lambda| .
$$
\end{Theorem}
The proof of Theorem \ref{th:ls} can be found in Section \ref{sec:ls}.
As a conclusion it is natural to ask whether one can find the precise
asymptotic behavior, as $q\downarrow 0$, of the log-Sobolev constant, and more generally of any $\alpha$-log-Sobolev inequality.

\subsection{Out of equilibrium I: long time behavior}

In this section we address the following questions: does the law of
the process at time $t$ converge to the reversible measure $\pi$ as
$t\to \infty$ if it starts from some non-equilibrium measure $Q\neq
\pi$? And if it converges, how fast?

As already explained in the previous section, one usually answers such questions studying the
log-Sobolev constant and using the so-called hypercontractivity
property of the semigroup $e^{t\cL}$.
Unfortunately, the log-Sobolev constant of a segment of size $L$ compares to $L$ as stated in Theorem \ref{th:ls}
(and in particular is not uniformly bounded in the size of the system) so that the (now) usual Holley-Stroock strategy does not apply.

Taking advantage of the oriented character of the East process, one can anyway prove the following result.

\begin{Theorem}[\cite{CMST}] \label{th:outI}
Fix $\a \in (0,1) \setminus \{p\}$ and assume that the initial
distribution $Q$ is a  Bernoulli($\a$) measure. Then for any local
function $f$,
  \begin{equation*}
    \int dQ(\sigma) |\bbE_\sigma(f(\s_t))-\pi(f)|\le C_f e^{-mt}
  \end{equation*}
  where $C_f= \|f\|_\infty (p \wedge q)^{-|\mathrm{supp}(f)|}/|q-\alpha|$
  (with $|\mathrm{supp}(f)|$   the cardinality of the support of $f$), and $m = \frac{1}{2} \gap (\cL) \min(1,\frac{\log (1/\alpha)}{\log(\alpha /p\wedge q)})$.
\end{Theorem}
The above result, proved in Section \ref{sec:outI},  shows that
relaxation to equilibrium is indeed taking place at an exponential
rate on a time scale $T_{\rm relax}=\gap(\cL)^{-1}$ which, for small
values of $q$, is very large and of the order of $\nep{c
\log(1/q)^2}$ with $c=(2\log 2)^{-1}$.
\begin{remark}
Although the above result is quite natural it should be noted that
it cannot hold for \emph{any} initial law $Q$ (as one could naively
expect). Consider for example starting the East dynamics from i.i.d
on the negative part of $\bbZ$ and from identically equal to one on
the positive part. Then clearly the positive part of $\bbZ$ will
never relax to equilibrium just because there are no vacancies
around! We refer to \cite{CMST} for a complete classification of the
allowed initial distributions.

In higher dimensional models like the
North-East model the non-equilibrium dynamics should exhibit an even
richer structure because of the possible presence of a critical density above
which the spectral gap becomes zero, infinite blocked configurations
appear etc. We refer the interested reader to the introduction of
\cite{CMST} for a quick review.

Finally we observe that the oriented character of the East Model is essential in the proof of Theorem
\ref{th:outI}. However the asymptotic convergence to the reversible measure should hold in the ergodic regime for more general KCSM. A particular step in this direction can be found in \cite{blondel}.

\end{remark}

\subsection{Out of equilibrium II: plateau behavior, aging and scaling limits} \label{sec:outIIintro}

In this section, we give a set of results (Theorem \ref{plateau} and
\ref{asymptotics}) which details the non-equilibrium behavior of the
East process for small values of $q$ (small temperature) and, in
contrast with  the previous section, for time scales much smaller
that $T_{\rm relax}=\gap(\cL)^{-1}$. The proofs of both theorems are
quite involved. Hence they will not be given in full detail. We
mention that one of the main ingredient is an approximation of the
East model by means of a suitable hierarchical coalescence process
introduced in the physics literature \cite{SE1} and rigorously
studied in \cite{FMRT} (see also \cite{FRT} for extensions). The
definition of this coalescence process and the approximation result
will not be given here but can be found in \cite{FMRT-cmp}.

\begin{definition}
\label{stalling-active}
Given $\epsilon,\, q\in (0,1)$,   we set
\begin{eqnarray}
& t_0:=1; ~~~~~t_0^-:=0; ~~~~~t_0^+=\left(\frac{1}{q}\right)^{\epsilon}\nonumber\\
& t_n:=\left(\frac{1}{q}\right)^n;~~~~~t_n ^-:=t_n ^{1-\e};~~~~~t_n
^+=t_n ^{1+\e}\,\,\,\,\,\,\forall n\geq 1\,.\label{deftn}
\end{eqnarray}
The time interval $[t_n^-, t_n^+]$ and $[t_n^+,t_{n+1}^-]$ will be
called respectively the {\sl $n^{th}$-active period} and the {\sl $n^{th}$-stalling
period}.
\end{definition}

In the next theorem we deal with the persistence, the vacancy
density and the two-time autocorrelations  during stalling periods
and prove plateau and aging behavior.

Given a configuration $\sigma$ we denote by $\{x_k\}=\{x_k(\sigma)\}
$ the position of the empty sites of $\sigma$, with the rules that
$x_0 \leq 0 < x_1$ and $x_k < x_{k+1}$ for all integers $k$. Then,
given a probability measure $\mu$ on $\bbN$, we write $Q={\rm
Ren}(\mu\tc 0)$ if, under $Q$, the first zero $x_0$ is located at
the origin $0$, the random variables $\{x_k - x_{k-1}\}_{k =
1}^\infty$ form a sequence of i.i.d.\ random variables with common
law $\mu$, and there is no other empty site on the left of the
origin.

\begin{Theorem}[Persistence, vacancy density and two-time autocorrelations]\label{plateau}
Fix a pro\-bability measure
$\mu$ on $\bbN$, $d := \inf \{ a :
\mu(a)>0\}$ and let $n_d$ be the smallest integer $n$ such that 
$d \in [2^{n-1}+1,2^n]$. Assume that the initial distribution $Q$ is a renewal
measure $Q={\rm Ren}(\mu\tc 0)$
and either one of the following holds:
\begin{enumerate}[a)]
\item the measure $\mu$ has finite mean;
\item the measure $\mu$ belongs to the domain of attraction of a $\a$-stable law
  or, more generally, $\mu((x,+\infty))=x^{-\a}L(x)$ where $L(x)$ is a
  slowly varying function at $+\infty$, $\a\in [0,1]$\footnote{A function $L$ is said to
be slowly varying at infinity, if, for all $c>0$, $\lim\limits_{x \to
\infty} L(cx)/L(x)=1$.}.
\end{enumerate}
Then, if $o(1)$ denotes an error term depending only
on $n$ and tending to
zero as both tend to infinity, for any $n\geq n_d$
\begin{enumerate}[(i)]
\item \begin{align}
& \lim_{q\downarrow 0}\, \sup _{t \in [t_{n}^+, t_{n+1}^-]}
\left|\bbP_Q(\sigma_t(0)=0)-\left(\frac{1}{2^{n}+1}\right)^{c_0(1+o(1)) }\right|=0\,,\\
& \lim _{q\downarrow 0}\, \sup _{t \in [t_{n}^+,
t_{n+1}^-]}\left|\bbP_Q(\sigma_s(0)=0~~~\forall s\leq
t)-\left(\frac{1}{2^{n}+1}\right)^{c_0(1+o(1)) }\right|=0\,,
\end{align}
where $c_0=1$ in case (a) and $c_0=\a$ in case (b).
\item Let $t,s:[0,1/2]\to [0,\infty)$ with $t(q)\geq s(q)$ for all $q\in [0,1/2]$. Then
$$\varlimsup_{q\downarrow 0}\,  \bbP_Q(\sigma_{t(q)}(0)=0)\leq \varlimsup _{q\downarrow 0}\,
 \bbP_Q(\sigma_{s(q)}(0)=0).$$
 The same bound holds with $\varliminf_{q\downarrow 0}$ instead
 of $\varlimsup_{q\downarrow 0}$.
\item For $x\in \bbZ_+$ let $C_Q(s,t,x)=\text{\rm Cov}_Q(\sigma_t(x);\sigma_s(x))$ be the two-time autocorrelation function. Then, for any $n,m \geq n_d$
$$
\lim _{q\downarrow 0}\, \sup _{\stackrel{t \in [t_{n}^+, t_{n+1}^-]}{s \in [t_{m}^+, t_{m+1}^-]}}
\left|C_Q(s,t,x)
  -\rho_x\left(\frac{1}{2^{n}+1}\right)^{c_0(1+o(1))}\left(1- \rho_x \left(\frac{1}{2^{m}+1}\right)^{c_0(1+o(1))}
    \right)\right|=0$$
where $\rho_x=Q(\s(x)=0)$.
\end{enumerate}
\end{Theorem}


The picture that emerges from points $(i)$ and $(ii)$ is depicted in Figure \ref{fig:plateau}.

\begin{figure}[ht]
\psfrag{T}{ $\!\!\!\! \frac{\log t}{|\log q|}$}
 \psfrag{0}{\footnotesize $\epsilon$}
 \psfrag{1}{\footnotesize $ 1-\epsilon$}
 \psfrag{2}{\footnotesize $1+\epsilon$}
 \psfrag{3}{\footnotesize$2-\epsilon$}
 \psfrag{4}{\footnotesize$2+\epsilon$}
 \psfrag{5}{\footnotesize$n+\epsilon$}
 \psfrag{6}{}
\psfrag{cn}{$c_n$}
\psfrag{p}{$\!\!\!\!\!\!\!\! \mathbb{P}_{\mathcal{Q}}(\sigma_t(0)=0)$}
 \includegraphics[width=.90\columnwidth]{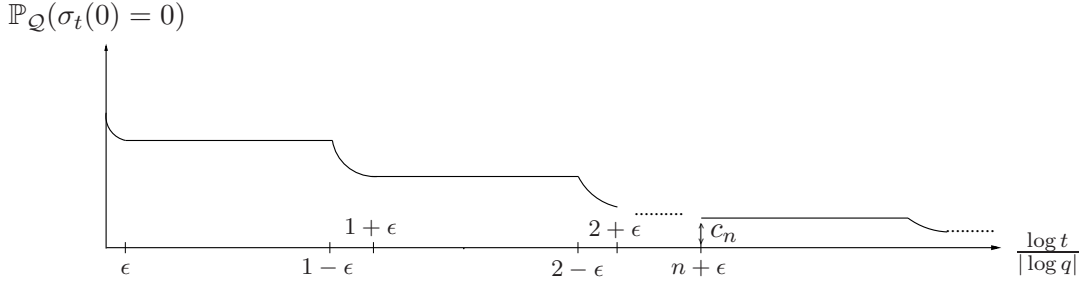}
\caption{Plateau behavior in the limit $q \to 0$ (with, for simplicity $d=1$), where  we set $c_n:=\left(1/(2^{n}+1)\right)^{c_0(1+o(1))}$ with $c_0$ defined in Theorem \ref{plateau} and $o(1)$ going to zero as $n\to\infty$.}
\label{fig:plateau}
 \end{figure}

\begin{remark} \label{anyk}
Theorem \ref{plateau} is stated in \cite{FMRT-cmp} in a less general form. Indeed, in \cite{FMRT-cmp},
$\mu$  must satisfy the following technical condition: for any $k\in \bbN$, $\mu\left([k,\infty)\right) >0$.
 In order to remove this assumption  we will need a new  technical
 fact described by Lemma \ref{chepalle} in  Section \ref{sec:outII}.
 Having such a lemma, it is simple to adapt the proof of
 \cite{FMRT-cmp}.  On the other hand, Theorem \ref{plateau} holds
 now only  starting from scale $n_d$ simply because, on any smaller scale, nothing interesting happens since the filled sites are essentially frozen.
 

Finally, we observe that, for small values of $q$, the two-time
autocorrelation function  $C_Q(s,t,x)$  depends in a non trivial way
on $s,t$ and not just on their difference $t-s$ (see Point $(iii)$).
This explains the word ``aging''. Clearly, for times much larger
than the relaxation time $\gap^{-1}(\cL)$, the time auto-correlation
will be very close to that of the equilibrium process which in turn,
by reversibility, depends only on $t-s$.
\end{remark}

The next theorem describes the statistics of the interval (domain)
between two consecutive zeros in a stalling period. In order to
state it let, for any $c_0\in (0,1]$, $\tilde X^{(\infty)}_{c_0}\ge
1$ be a random variable with Laplace transform given  by
\begin{equation}\label{macedonia}
\bbE(\nep{-s\tilde X^{(\infty)}_{c_0}}) =1- \exp\Big \{ - c_0 \int _1^\infty
\frac{e^{-sx}}{x} dx \Big\}= 1-\exp\Big \{ - c_0 \, \text{Ei}(s)
\Big\}\,.
\end{equation}
The corresponding probability density is of the form (see
\cite{FMRT}) $p_{c_0} (x){\mathds{1}}_{x \geq 1}$ where $p_{c_0}$ is the
continuous function on $[1, \infty)$ given by
\begin{equation}\label{vonnegut}
 p_{c_0}(x)= \sum _{k=1}^\infty \frac{(-1)^{k+1}c_0^k}{k!}\, \rho_k(x)
 \mathds{1}_{x\geq k }\,,
 \end{equation}
 where $ \rho _1(x)= 1/x$ and
 \begin{equation}\label{kurt}
   \rho_{k+1} (x)= \int _1^\infty d x_1
  \cdots \int_1 ^\infty dx _k \frac{1}{ x-\sum_{i=1}^{k} x_i } \prod
 _{j=1}^{k} \frac{1}{x_j} \,,\qquad  k\geq 1\,.
 \end{equation}
Let also $\tilde Y^{(\infty)}_{c_0}$ be
a non-negative random variable with Laplace transform given  by
\begin{equation}\label{macedonia2}
\bbE(\nep{-s \tilde Y^{(\infty)}_{c_0}}) :=1- \exp\Big \{ - c_0 \int _0^1
\frac{e^{-sx}}{x} dx \Big\}
\end{equation}

Starting from an initial law $Q={\rm Ren}(\mu\tc 0)$ denote by $x_0(t)=x_0(\sigma_t)$ the position of
the first zero at time $t$, and by $x_1(t)=x_1(\sigma_t)$ the position of the second zero.

\begin{Theorem} [Limiting behavior of the domain length and of the first zero]
\label{asymptotics}
Under the same assumptions of Theorem \ref{plateau}, let
$$
\bar X^{(n)}(t):=(x_{1}(t)-x_0(t))/(2^{n-1}+1)\quad ;\quad \bar
Y^{(n)}(t):=x_0(t)/(2^{n-1}+1).
$$
 Then, for any bounded function $f$,
 \begin{align}
&\lim _{n\uparrow \infty}
 \lim_{q\downarrow 0}  \sup_{t \in [t_n^+, t_{n+1}^- ]}
 \Bigl| \bbE_Q\bigl( f( \bar  X^{(n+1)}(t) ) \bigr)- E\bigl(f(\tilde
 X_{c_0}^{(\infty)}\bigr)\Bigr|=0  \label{spa}\\
&\lim _{n\uparrow \infty}\lim_{q\downarrow 0}  \sup_{t \in [t_n^+, t_{n+1}^-]}
 \Bigl| \bbE_Q\bigl( f( \bar  Y^{(n+1)}(t) ) \bigr)- E\bigl(f(\tilde
 Y_{c_0}^{(\infty)}\bigr)\Bigr|=0
\label{spaghetti1}
\end{align}
where again $c_0=1$ if $\mu$ has finite mean and $c_0=\a$ if  $\mu$
belongs to the domain of attraction of a $\a$-stable law.\\
The result \eqref{spa} holds for $f$ satisfying $|f(x)|\leq C
(1+|x|)^m$, $m=1,2,\dots$, if the $(m+\d)^{th}$-moment of $\mu$  is
finite for some $\d>0$.
\end{Theorem}
\begin{remark}
The above result holds for a wider class of initial measure $Q$.
Moreover, the moment condition can be relaxed, see \cite{FMRT-cmp}.
The proof of the above theorem will not be given here and can be
found in \cite{FMRT-cmp}. In general, the asymptotics of the first
$k$ zeros can be deduced from the results of \cite{FMRT-cmp}.
\end{remark}
We refer the reader to Remark \ref{rem:comb} in Section
\ref{sec:comb} for an heuristic interpretation of the scaling
$t_n=(1/q)^n$ and the renormalization length $2^n$.

\subsection{Large deviations of the activity}
Let us introduce some notation. For simplicity, in this section, we
set $\Lambda_N=[1,N]$,  $\sigma_t^N=\sigma_t^{\Lambda_N}$ and
$\pi_N:=\pi_{\Lambda_N}$ (recall the notation introduced two lines after \eqref{estendiamo}) 
and we denote by $\langle \rangle$ the mean
over the evolution of the process and over   the initial
configuration which is distributed with $\pi_N$. We also define the
total activity as
$$
\cA(t):= \sum_{x \in \Lambda_N} \cA_x(t)
$$
where $\cA_x(t):=\# \{s \leq t : \lim_{\e \downarrow 0} \sigma_{t-\e}^N(x) \neq \lim \sigma_t^N(x) \}$
is the random variable
that counts the number of configuration changes at site $x$ during the time interval $[0,t]$.

Let us explain why the total activity is a relevant quantity. The East model, as it is common for kinetically constrained models and more generally for glassy systems, is characterized by
a spatially heterogeneous dynamics, namely by the mixture of frozen and mobile areas (see for example section 1.5 in \cite{GarrahanSollichToninelli}). The occurrence of these heterogeneities has led to the idea that
the dynamics takes
place on a first-order coexistence line between active and inactive dynamical phases
 \cite{MGC,GJLPDW1,GJLPDW2,GarrahanSollichToninelli}. In order to exploit this idea the activity has been proposed as a relevant order parameter to discern active and inactive dynamics and a dynamical approach has been devised to define a suitable notion of free energy. In this dynamical approach the role of the free energy is played by the large deviation function of the activity which, as will be detailed below, undergoes a first order transition in the thermodynamic limit.

Since   $\cA_x(t) - \int_0^t c_x^{\Lambda_N}(\sigma_s^N) ds$ is a
martingale, it can be proved that $\cA(t)$ satisfies the following
law of large numbers
$$
\lim_{N \to \infty} \lim_{t \to \infty} \frac{\cA(t)}{Nt} =  2p(1-p)^2 .
$$
In the sequel we set $\bbA := 2p(1-p)^2$.
Thus  one could expect that the probability $\mathcal{P}(a)$ of observing a deviation from the mean value $\frac{\cA(t)}{N t} \simeq a$ scales as
\begin{equation}
\label{eq: LD intro}
\lim_{N\to\infty}\lim_{t\to\infty} \;  \frac{1}{Nt} \log \mathcal{P}(a)=
-f(a) \, ,
\end{equation}
with $0<f (a) <\infty$ for $a\neq \mathbb A$ as it occurs in absence of the kinetic constraint. However, as it has been observed in \cite{GJLPDW1,GJLPDW2}, due to the presence of the constraint  it is possible to realize at a low cost a trajectory with zero
activity by starting from a completely filled configuration and imposing that a single site does not change its state. Analogously one can obtain an  activity  smaller than the mean one by blocking for a fraction of time a  single site. As a consequence of this sub-extensive cost for lowering the activity it holds
$f(a)=0$ for $a<\mathbb A$.
For the same reason, the moment generating function controlling the fluctuation of the total activity
\begin{eqnarray}
\label{eq: 1st order intro}
\psi(\lambda) = \lim_{N \to \infty} \lim_{t \to \infty} \;
\frac{1}{N t} \log \left \langle \exp \big( \lambda \cA(t) \big) \right \rangle
\end{eqnarray}
is non analytic at $\lambda=0$ with a discontinuous first order derivative \cite{GJLPDW1,GJLPDW2}.

In \cite{bodineau-toninelli} the authors study the finite size scaling of this first order transition by analyzing this generating function
with a refined thermodynamic scaling, namely
$$
\varphi(\alpha) := \limsup_{N \to \infty} \lim_{t \to \infty} \frac{1}{t} \log \left\langle \left( \exp \{ \frac{\alpha}{N} \cA(t) \}\right)\right\rangle, \qquad \alpha \in \bbR
$$
which corresponds to a blow up of the region $\lambda=\alpha/N \sim 0$ and prove the following result

\begin{Theorem}[\cite{bodineau-toninelli}] \label{th:activity}
For any $p\in(0,1)$, there exist $\alpha_1 < \alpha_0<0$ and a constant $\Sigma>0$ such that
\begin{enumerate}[(i)]
\item for $\alpha > \alpha_0$, it holds $\varphi(\alpha)=\bbA \alpha$;
\item for $\alpha < \alpha_1$ it holds $\varphi(\alpha)=- \Sigma$.
\end{enumerate}
\end{Theorem}
The  results of this theorem are illustrated in Figure \ref{fig:activity} below. In particular it shows that, in a range of $\lambda $ of order $1/N$, the transition is shifted from 0.
As a corollary, the authors give the following estimates on the large deviations for a reduced activity

\begin{Corollary}[\cite{bodineau-toninelli}]\label{cor:activity}
For any $u \in [0,1)$, it holds
$$
\begin{array}{rcccl}
-\Sigma(1-u) \!\!& \leq & {\displaystyle \lim_{\e \to 0} \liminf_{N \to \infty} \lim_{t \to \infty}}
 \frac{1}{t} \log \left\bra\left( \frac{\cA(t)}{Nt} \in [u\bbA -\e, u \bbA+\e]\right)\right\ket & & \\
&\leq  & {\displaystyle \lim_{\e \to 0} \limsup_{N \to \infty} \lim_{t \to \infty}}
 \frac{1}{t} \log \left\bra\left( \frac{\cA(t)}{Nt} \in [u\bbA -\e, u \bbA+\e]\right) \right\ket\!\! & \leq & \alpha_0 \bbA (1-u).
\end{array}
$$
\end{Corollary}
This scaling is  anomalous compared to the extensive scaling in $N$ of the unconstrained model and it is a direct consequence of the  sub-extensive cost for lowering the activity (while the large deviations for increasing the activity above $\bbA$ remain extensive in $N$). As detailed in Section 2.2 of \cite{bodineau-toninelli},
by analogy with equilibrium phase transitions, one can interpret $\Sigma$ as a surface tension between the inactive and the active region (per unit of time)
and this quantity is obtained from a variational problem (see Section 6 of \cite{bodineau-toninelli}).
\begin{figure}[ht]
\psfrag{psi}{ $\! \psi(\lambda)$}
 \psfrag{a0}{$\alpha_0$}
 \psfrag{a1}{$\alpha_1$}
 \psfrag{l}{ $\lambda$}
 \psfrag{phi}{$\varphi(\alpha)$}
 \psfrag{a}{$\alpha$}
 \psfrag{e}{$-\Sigma$}
 \includegraphics[width=.70\columnwidth]{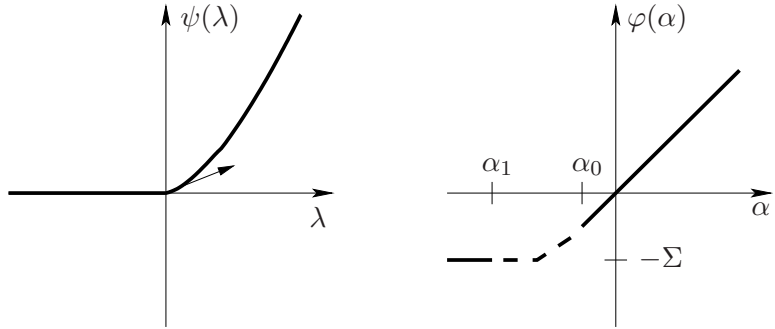}
\caption{The functions $\psi$ and $\varphi$. The results of Theorem \ref{th:activity} are depicted on the right, in thick line. The dashed lines correspond to a conjectured behavior.}
\label{fig:activity}
\end{figure}

The results in \cite{bodineau-toninelli}  do not provide  the entire  phase diagram for the generating function $\varphi(\alpha)$ (only for $\alpha \not \in [\alpha_1,\alpha_0]$).
 A first open problem is to prove the conjecture of \cite{bodineau-toninelli} that there is a unique critical value $\alpha_c$ and that the two regimes remain valid up to $\alpha_c$ (as depicted in the dashed line of figure \ref{fig:activity}), namely $\varphi=-\Sigma$ for $\alpha\leq \alpha_c$ and $\varphi=\alpha\bbA$ for $\alpha\geq\alpha_c$, which would imply by continuity
$
\alpha_c =-\Sigma/\bbA.$
This conjecture is supported by numerical simulations \cite{BLT}.
If this  conjecture is verified, then Corollary \ref{cor:activity}
 can be improved and the large deviations for reducing the activity
would be given by
$$
\forall u \in [0,1], \qquad
\lim_{N\to\infty} \lim_{t\to\infty}\frac{1}{t}\log \left \bra \frac{\cA(t)}{Nt}\simeq u \bbA \right \ket  =-\Sigma(1-u) \, .
$$
Another interesting open problem is to describe the interface between the active and inactive regions in the inactive regime  $\alpha<\alpha_c$, in particular to prove that this interface is localized close to the boundary. Indeed in \cite{BLT} the interface has been conjectured to fluctuate within a region of the order of $N^{1/3}$, a result which is supported by numerical results and by the study of an effective model where the boundary between active and inactive
regions is described by a Brownian interface. Finally, a very interesting issue from the physicists point of view is to understand if and how the dynamical phase transition can lead to  quantitative predictions on the model at $\lambda = 0$.

\section{Combinatorics} \label{sec:comb}

In this section we collect, and partially prove, some useful
combinatorial results, which are  intimately related to the oriented
character of the East process.

Let us fix some notation. Consider the East process on $\Lambda=(-\infty,-1]$ starting from the completely filled configuration ({\it i.e.}\
 the configuration $\sigma_\mathds{1}$ satisfying $\sigma_\mathds{1}(x)=1$ for any $x \in (-\infty,-1]$) with a frozen vacancy at the origin $0$. Denote by
$V(n)$ the set of all configurations that the process can reach (with positive probability) under the condition that, at any given time, no more than $n$ zeros are present on $(-\infty,-1]$.
Define absolute value of the position of the leftmost zero in all the configurations of $V(n)$, namely
$$
\ell(n):=\sup\left\{-y: y \in (-\infty,-1], \exists \eta\in V(n)\, {\mbox{s.t.}}\,\eta(y)=0\right\}
$$
with the convention that $\sup\emptyset=0$.

\begin{Proposition}[\cite{CDG,SE1}] \label{prop:comb}
The following holds for $n \geq 0$:\\
$(i)$ $\ell(n)=2^{n}-1$;\\
$(ii)$ $|V(n)| \leq 2^{\genfrac{(}{)}{0pt}{}{n}{2}} n!c^n$ with $c \leq 0,7$.
\end{Proposition}

\begin{remark} \label{rem:comb}
The result of Point $(i)$ holds true if one replaces $\Lambda$ by
any finite interval  $[-L,-1]$ with $L \geq 2^n-1$.

More precise statements on the cardinality $|V(n)|$ of $V(n)$ (and
in particular a lower bound of the same order with a different
constant $c$) can be found in \cite{CDG}.

Consider the East process on $[-\ell, -1]$ with a frozen zero at the
origin $0$. Let $n\in \bbN$ be such that  $\ell \in [2^{n-1}+1,2^n]$
and $n \geq 1$, or $\ell=1$ and $n=0$. Suppose the East process
starts from the configuration having a single vacancy located at
$-\ell$ (\ie starts from $01111\cdots1$). Then, due to the above
proposition,  the system must create at least $n$ extra zeros
 in order to kill the vacancy at $-\ell$.
 This
occurs with probability of order $q^n$ and thus with  an activation
time $t_n=(1/q)^n$. This explains the scale $t_n$ introduce in
Section \ref{sec:outIIintro}, and the renormalization length $2^n$.
\end{remark}

Point $(i)$, stated in  \cite[Fact 1 (i)]{CDG}, had already been
noticed in \cite{SE2},\cite[Section B]{SE1}. In order to be self
contained and to clarify a mechanism which will be at the heart of
the behavior of the  East process  when $q\downarrow  0$, we provide
its proof. With respect to \cite{CDG}, we present a slightly
different approach based on the arguments of \cite{SE1} .

The proof of Proposition \ref{prop:comb} is based on the following
lemma:
\begin{Lemma}
 Let $V(n,k)$, for $0 \leq k\leq n$, be the subset of $V(n)$ given by
those configurations which contain exactly $k$ zeros. Let
 $$
 \tilde \ell(n,k):=\sup\left\{|y|: y \in (-\infty,-1], \exists \eta\in V(n,k)\, {\mbox{s.t.}}\,\eta(y)=0\right\}
 $$
 be the position of the leftmost zero considering  all the configurations of
 $V(n,k)$.

 Then the following holds:
 \begin{itemize}

 \item[(i)]  $\tilde\ell(n,k)$ is increasing in $k$ for $n\geq k \geq 0$,

 \item[(ii)] $\ell(n)=\tilde\ell(n,n)$ for $n\geq 0$,

\item[(iii)] $\tilde\ell(n,1)=\tilde\ell(n-1,n-1)+1$ for $n \geq 1$,

\item[(iv)]  $\tilde\ell(n,k)\geq\sum_{j=1}^{k}\tilde\ell(n-j+1,1)$ for $n \geq k \geq 1$.
\end{itemize}

\end{Lemma}
\begin{proof}
To prove Item (i)
 we  only need to
 exhibit, for $k<n$,
  a configuration in $V(n,k+1)$ with
a zero at  $-[\tilde\ell(n,k) +1]$, \ie a configuration with $k+1$
zeros,  one of which at $-[\tilde\ell(n,k) +1]$, that the system can
reach from $\s_\mathds{1}$ without exceeding quota  $n$ zeros. Such
a configuration can be obtained as follows:  starting from
$\s_{\mathds{1}}$ the system reaches  the configuration which
realizes $\tilde\ell(n,k)$ without exceeding quota  $n$ zeros. This
configuration has $k<n$ zeros, therefore  the system is allowed to
  create  an additional zero at  $-[\tilde\ell(n,k) +1]$.

From the fact that $\tilde\ell(n,k)$ is increasing in $k$, we get
the identity $\tilde\ell(n,n)=\ell(n)$ stated in Item (ii).

To prove Item (iii) we observe that, just before creating the zero
at $-\tilde\ell(n,1)$,  the system should have a zero at
 $-\tilde\ell(n,1)+1$. Moreover, after creating  the zero at
 $-\tilde\ell(n,1)$, the system should  remove all zeros on the
 right of $-\tilde \ell (n,1)$  without exceeding quota $n$ zeros  and
 therefore without exceeding quota $n-1$ zeros on $[- \tilde \ell
 (n,1)+1, -1]$. Reversing this last part of the evolution and disregarding what happens outside
 $[- \tilde \ell
 (n,1)+1, -1]$, we get a trajectory starting  from  the fully filled
 configuration $\s_\mathds{1}$ and realizing the zero at $ - \tilde \ell
 (n,1)+1$ without exceeding quota $n-1$ zeros.  The conclusion
 then follows by applying Item (i).

It remains to prove Item (iv).  In order to realize the path which
leads to a configuration with $k$ zeros, one of which at
$-\tilde\ell(n,k)$ and without exceeding quota $n$ zeros, the system
can  first create a single zero at $-\tilde\ell(n,1)$ without
exceeding quota $n$ zeros, then it can use this zero  as an anchor
to create a further zero  at $-\tilde\ell(n,1)-\tilde\ell(n-1,1)$ by
means of a path with  at most $n$ simultaneous zeros including the
one at $-\tilde\ell(n,1)$. By continuing this procedure we get the
inequality of Item (iv).
\end{proof}

We can  now come back to Proposition \ref{prop:comb}:

\smallskip

\noindent {\sl Proof of Prop. \ref{prop:comb}}. Combining Items
(iii) and (iv) of  the above lemma,
 we get
 $$
 \tilde\ell(n,n)\geq\sum_{j=1}^n\tilde\ell(n-j,n-j)+n = \sum
 _{j=0}^{n-1}\tilde \ell (j,j) +n\,.
 $$
 Then, using the fact that $\tilde\ell(0,0)=0$, we get by induction that
 $$
 \tilde\ell(n,n)\geq 2^n-1 .
 $$
 Using Item (ii) of the above lemma, we conclude that $\tilde \ell
 (n) \geq 2^n-1$.

 The proof of the reverse inequality $\tilde\ell(n)=\tilde\ell(n,n) \leq 2^n-1$ goes by
 induction. It is trivially fulfilled for $n=0$.
 Suppose the inequality holds up to $n-1$, where $n \geq 1$. Then, Item (iii) in the above Lemma
  implies that $\tilde \ell(k,1)\leq 2^{k-1}$ for all $k: 1\leq k \leq n$.
 Then consider a configuration $\eta$ which realizes $\tilde\ell(n,n)$. By definition $\eta$  should contain $n$ zeros and
 by using reversibility we know there should exists a path which kills all these zeros with at each step at most $n$ zeros.
 Thus at least one of the $n$ zeros of $\eta$ should be such that it can be killed without creating further zeros, namely it should have the next zero to the right at distance $\tilde\ell(1,1)=1$. Let for simplicity $\eta$ contain only one such zero (otherwise the
 strategy can be easily adapted) and consider the configuration $\eta'$ obtained by killing this zero. Then $\eta'$
  should contain at least a zero which has the next zero at   distance at most  $\tilde\ell(2,1)$
   (it should be killed by adding at most one extra zero) and the configuration obtained by killing this zero should contain at
    least a zero which has the next zero  at distance at most $\tilde\ell(3,1)$ and so on. By using this observation and the iterative assumption which guarantees that  $\tilde\ell(k,1)\leq 2^{k-1}$ for all $k\leq n$ we get
$$\tilde\ell(n,n)\leq 1+2^1+\dots 2^{n-1}=2^n-1$$ which, together with the above lower bound leads to the desired result $\ell(n)=\tilde\ell(n,n)=2^n-1$.

The proof of Point $(ii)$ can be found in \cite[Theorem 2, 4 and
5]{CDG}. \qed


\section{Spectral gap: proof of Theorem \ref{th:gap}} \label{sec:coercive}

In this section we prove Theorem \ref{th:gap}.
We start with the proof of Point $(iii)$ which is the easiest.
Actually we will prove the following stronger useful result
(from which Point $(iii)$ immediately follows, details are left to the reader).

\begin{Proposition}[Monotonicity of the spectral gap \cite{CMRT}]\label{prop:monotony}
Let $V \subset \Lambda$ be two intervals of $\bbZ$ (possibly infinite). Then
$$
\gap(\cL_\Lambda) \leq \gap(\cL_V).
$$
\end{Proposition}

\begin{proof}
For any local function $f: \O_V \to \bbR$ we have
$\Var_V(f)=\Var_\L(f)$ because of the product structure of the
measure $\pi_\L$ and $\cD_\L(f)\le \cD_V(f)$ because, for any $x\in
V$ and any $\sigma \in \O_{\L}$, $c_{x}^{\L}(\sigma) \le
c_{x}^{V}(\sigma)$. The result follows at once from the variational
characterization of the spectral gap.
\end{proof}

Point $(i)$ and the upper bound in Point $(ii)$ of Theorem
\ref{th:gap} follow from the next result. We use the following
standard notation: $\log_2 a=\log a/\log 2$, $a>0$.

\begin{Theorem}[\cite{CMRT}]\label{th:cmrt-gap}
For any $\d\in (0,1)$ there exists $C_\d>0$ such that
\begin{equation} \label{eq:th1}
\gap(\cL) \ge q^{C_\d} q^{\log_2(1/q)/(2-\d)} .
\end{equation}
In particular
\begin{equation}
\limsup_{q\to 0} \log(1/\gap(\cL))/(\log(1/q))^2 \leq \left(2\log 2\right)^{-1} .
\label{eq:th2}
\end{equation}
\end{Theorem}

The lower bound in Point $(ii)$ of Theorem \ref{th:gap} is proven in
\cite{Aldous} by a subtil but somehow intricate argument based on
path techniques. In \cite[Theorem 5.1]{CMST} the authors give an
alternative (softer) proof. Below  we reproduce such proof in order
to clarify the  role played by  energy barriers.

\begin{Theorem}[\cite{CMST}] \label{th:cmst-gap}
For any $\delta \in (0,1)$, there exists $C_\delta>0$ such that
\begin{equation*}
\gap(\cL) \le C_\delta q^{\log_2(1/q)(1-\delta)/2} .
\end{equation*}
In particular
\begin{equation*}
\liminf_{q\to 0} \log(1/\gap(\cL))/(\log(1/q))^2 \geq \left(2\log 2\right)^{-1} .
\end{equation*}
\end{Theorem}

\begin{proof}[Proof of Theorem \ref{th:cmrt-gap}]
The limiting result \eqref{eq:th2} follows at once from \eqref{eq:th1}.

In order to get the lower bound \eqref{eq:th1} we will apply the
bisection-constrained method introduced in \cite{CMRT}, which
extends the classical bisection method \cite{SFlour}. Observe first
that, due to Proposition \ref{prop:monotony},
 $\gap(\cL) \geq \inf_{\Lambda}
\gap(\cL_\Lambda)$ where the infimum runs over all possible finite
intervals $\Lambda = [a,b] \subset \mathbb{Z}$.  Hence, our aim is
to prove a lower bound on $\gap(\cL_\Lambda)$, uniformly in
$\Lambda$.

 Fix
$\d\in (0,1/6)$ and define $l_k=2^k$, $\d_k=\inte{l_k^{1-3\d}}$ and
$s_k:=\inte{l_k^{\d}}$, $\inte{\cdot}$ denoting the integer part.
Let also $\bbF_k$ be the set of intervals which, modulo
translations, have the form $[0,\ell]$ with $\ell\in
[0,l_k+l_k^{1-\d}]$. As in \cite{cesi}, given $\Lambda = [a,b] \in
\bbF_k \setminus \bbF_{k-1}$, for $i=1,\dots,s_k$ we define
$\Lambda_1^{(i)}:=[a,\frac{b+a}{2} + \frac{2i}{8} \delta_k]$ and
$\Lambda_2^{(i)}:=[\frac{b+a}{2} + \frac{2i-1}{8} \delta_k, b]$ so
that the sequence $\{\L_1^{(i)}, \L_2^{(i)}\}_{i=1}^{s_k}$ satisfies
the following properties:
\begin{enumerate}[(i)]
\item $\L = \L_1^{(i)} \cup \L_2^{(i)}$,
\item  $d(\L\setm \L_1^{(i)}, \L\setm \L_2^{(i)}) \ge \d_k/8 $,
\item $\left(\L_1^{(i)}\cap \L_2^{(i)}\right)\cap \left(\L_1^{(j)}\cap \L_2^{(j)}\right) = \emp$,  if $i\ne j$
\item $\L_1^{(i)}, \L_2^{(i)} \in \bbF_{k-1}$.
\end{enumerate}
Above, $d(\cdot, \cdot)$ denotes the
Euclidean distance.
\begin{remark}
In other words, any set of $\bbF_k \setminus \bbF_{k-1}$
can be obtained as a "slightly overlapping union" of two
intervals in $\bbF_{k-1}$.
\end{remark}
Define
$$
\gamma_k = \sup_{\Lambda \in \bbF_k} \gap(\cL_\Lambda)^{-1} .
$$
Due to   Proposition \ref{prop:monotony} the above supremum  is
attained on the intervals $\L_k=[a,a+l_k+l_k^{1-\d}]$. Applying the
bisection--constrained method introduced in \cite{CMST}, we want to
 establish  a   recursive inequality between $\gamma_k$ and
 $\gamma_{k-1}$. To this aim we
fix $\L\in \bbF_k\setm \bbF_{k-1}$ and write it as $\L = \L_1 \cup
\L_2$ with $\L_1, \L_2 \in \bbF_{k-1}$ as above (we drop the
superscript $(i)$, recall that $\L_1$ is on the left of $\L_2$).
Moreover, we set $I\equiv \L_1 \cap \L_2$. We now run the following
constrained ``block dynamics'' on $\O_\L$ with blocks
$B_1:=\L\setminus \L_2$ and $B_2:=\L_2$ (see Figure
\ref{fig:bisezione} below). The block $B_2$ waits a mean one
exponential random time and then the current configuration inside it
is refreshed with a new one sampled from $\pi_{\L_2}$. The block
$B_1$ does the same but now the configuration is refreshed only if
the current configuration $\s$ contains \emph{at least} one zero
inside the strip $I$.

The Dirichlet form of this auxiliary
chain is simply
$$
\cD_{\rm block}(f)=\pi_\L\left(c_1\Var_{B_1}(f)+\Var_{B_2}(f)\right)
$$
where $c_1(\sigma)$ is just the indicator of the event that $\sigma$
contains at least one zero inside the strip $I$ (as an illustration,
$c_1(\sigma)=1$ for the configuration given in Figure
\ref{fig:bisezione}). Recall that  $\Var_{B_1}(f)$, $\Var_{B_2}(f)$
depends only on  $\s_{B_1^c}$, $\s_{B_2^c}$ respectively.

Denote by $\g_{\rm block}(\L)$ the inverse spectral gap of the
auxiliary Markov chain on $\O_\L$ with block dynamics. The following
bound, whose proof can be found in \cite{CMRT} and at the end of
this section for completeness, is not difficult to prove.

\begin{Proposition} [\cite{CMRT}] \label{gapblock}
Let $\e_k \equiv {\displaystyle \max_{I}} \{\pi(\forall x \in I,
\sigma(x)=1) \}=  p^{\min_{I} |I|} $ where the $\max_{I}$ and the
$\min_{I}$ are taken over the $s_k$ possible choices of the pairs
$\left(\L_1,\L_2\right)$ and $I= \L_1 \cap \L_2$. Then
$$
\g_{\rm block}(\L) \le \frac{1}{1-\sqrt{\e_k}} .
$$
\end{Proposition}

As a consequence of the above result,
 by  writing down the standard Poincar\'e inequality
for the block auxiliary chain, we get  for any $f:\O_\L \to
\mathbb{R}$ that
\begin{equation}
  \label{eq:s1}
\Var_\L(f)\le \frac{1}{1-\sqrt{\e_k}}
\pi_\L\Bigl(c_1\Var_{B_1}(f)+\Var_{B_2}(f)\Bigr)\, .
\end{equation}
The second term, using the definition of $\g_{k-1}$ and the fact
that $B_2\in \bbF_{k-1}$, is bounded from above by
\begin{equation}
  \label{eq:s2}
  \pi_\L\Bigl(\Var_{B_2}(f)\Bigr)\le \g_{k-1}\sum_{x\in
    B_2}\pi_\L\Bigl(c_{x}^{B_2} \Var_x(f)\Bigr)\,.
\end{equation}
Notice that, by construction, $c_{x}^{B_2}(\s)=c_{x}^{\L}(\s)$ for
all $x\in B_2$ and all $\s$. Therefore the term $\sum_{x\in
  B_2}\pi_\L\bigl(c_{x}^{B_2} \Var_x(f)\bigr)$ is nothing but the
contribution carried by the set $B_2$ to the full Dirichlet form $\cD_\L(f)$.

Next we examine the more complicate term
$\pi_\L\bigl(c_1\Var_{B_1}(f)\bigr)$ with the goal in mind to bound
it with the missing term of the full Dirichlet form $\cD_\L(f)$.
Assume that $I=[a,b-1]$. For any configuration $\sigma$, define the
random variable $\xi$ as the the distance between  the rightmost
empty site in the strip $I$ and  the right boundary of $I$, namely
$$
\xi (\sigma) := \min_{x \in I: \sigma(x)=0} \{ b-x \}
$$
with the convention that $\min( \emptyset)=+ \infty$ (see Figure
\ref{fig:bisezione}). The indicator function $c_1$ guarantees that,
for any configuration $\sigma$ with $c_1(\sigma)=1$, $\xi(\sigma)
\in [1,b-a]$.
\begin{figure}[h]
\psfrag{l1}{ $B_1=\L\setminus \L_2$}
 \psfrag{l2}{$B_2=\L_2$}
 \psfrag{i}{$I$}
 \psfrag{xi}{$\xi$}
 \psfrag{bk}{$B_k$}
 \includegraphics[width=.90\columnwidth]{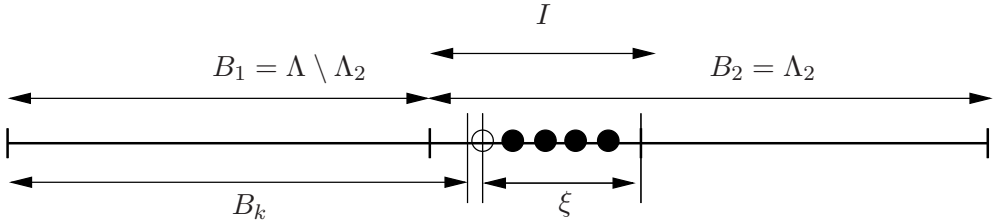}
\caption{The set $\Lambda$ divided into the blocks $B_1$ and $B_2$
and the random variable $\xi$. The configuration is such that
$c_1(\sigma)=1$.   Empty circles correspond to empty sites.}
\label{fig:bisezione}
 \end{figure}

Set for simplicity  $\ell :=b-a$ and decompose $c_1(\sigma)$
according to $\xi$ so that (observe that $\Var_{B_1}(f)$ depends
only on $\s_{B_2}$, such a dependence is understood below)
\begin{align}
 \pi_\Lambda \Bigl(c_1& \Var_{B_1}(f)\Bigr)  =
\sum_{k=1}^\ell \pi_\Lambda \Bigl(\mathds{1}_{\{\xi=k\}} \Var_{B_1}(f)\Bigr) \nonumber \\
& = \sum_{k=1}^\ell \sum_{\s_{B_2 \setminus I}} \pi_\Lambda (\s_{B_2
\setminus  I}) \sum_{\s_I} \pi_\Lambda (\s_I) \mathds{1}_{\{\xi=k\}}(\s_I)
\Var_{B_1}(f)
\nonumber \\
& = \sum_{k=1}^\ell \sum_{\s_{B_2 \setminus I}}
\pi_\Lambda(\s_{B_2\setminus I}) \sum_{\s_{I\setminus
I_k}}\pi_\Lambda(\s_{I\setminus I_k})\mathds{1}_{\{\xi=k\}}(\s_{I\setminus
I_k}) \sum_{\s_{I_k}}\pi_\Lambda(\s_{I_k})\Var_{B_1}(f) \label{A}
\end{align}
where $I_k:= [a,b-k-1]$. In the last identity    we used the
independence of $\mathds{1}_{\{\xi=k\}}$ from $\s_{I_k}$ (this comes from
the fact that $\xi$ is the rightmost empty site inside $I$, hence,
in order to decide that $\xi(\sigma)=k$, one has to know $\sigma(x)$
only for $x \in I \setminus I_k$).

Set $B_k=B_1 \cup I_k$ (see Figure \ref{fig:bisezione}). Then, the
convexity of the variance implies, for any $k$, that
\begin{equation}\label{superiore}
\sum_{\s_{I_k}}  \pi_\Lambda (\s_{I_k})\Var_{B_1}(f) \le
\Var_{B_k}(f) \,.
\end{equation}
 Then, the Poincar\'e inequality together with Proposition
\ref{prop:monotony} finally gives
\begin{align}
\Var_{B_k}(f)
& \le
\gap(\cL_{B_k})^{-1}\sum_{x\in B_k} \pi_{B_k}\bigl (c_{x}^{B_k} \Var_{x}(f) \bigr) \nonumber \\
& \le \gap( \cL_{\L_1})^{-1} \sum_{x\in
B_k}\pi_{B_k}\bigl(c_{x}^{B_k} \Var_{x}(f) \bigr) . \label{B}
\end{align}
Recall that $B_1= \L\setminus \L_2$, $B_2= \L_2$, $I= \L_1\cap
\L_2$.  The role of the event $\{\xi=k\}$ should at this point be
clear. Indeed, thanks to the empty site given by $\xi$, we have
\begin{equation} \label{C}
c_{x}^{B_k} (\s)\mathds{1}_{\{\xi=k\}}(\s) \leq c_{x}^{\Lambda}(\s)
\mathds{1}_{\{\xi=k\}}(\s)  \qquad \qquad \forall x \in B_k\,, \; \s \in
\O_\L\, .
\end{equation}
Let us come   back to \eqref{A}. Using \eqref{superiore}, \eqref{B},
\eqref{C} we conclude that the last member of \eqref{A} is bounded
from above by
\begin{multline*}
 \sum_{k=1}^\ell \sum_{\s_{B_2 \setminus I}}
\pi_\Lambda(\s_{B_2\setminus I}) \sum_{\s_{I\setminus
I_k}}\pi_\Lambda(\s_{I\setminus I_k})\mathds{1}_{\{\xi=k\}}(\s_{I\setminus
I_k})\gap( \cL_{\L_1})^{-1} \sum_{x\in B_k}\pi_{B_k}\bigl(c_{x}^{\L}
\Var_{x}(f) \bigr) = \\\gap(\cL_{\L_1})^{-1} \sum_{k=1}^\ell \pi_\L
\bigl( \mathds{1}_{\{\xi=k\}} \sum_{x\in B_k}c_{x}^{\L}\Var_{x}(f)\bigr)\,.
  \end{multline*}
Since $B_k=B_1 \cup I_k \subset B_1 \cup I = \L_1 \in \mathbb{F}
_{k-1}$, from the above bound and \eqref{A} we conclude that
\begin{align}
\pi_\L\Bigl(c_1\Var_{B_1}(f)\Bigr)
 \le
\g_{k-1}\,\pi_\L\Bigl(\sum_{x\in \L_1}c_{x}^{\L}\Var_{x}(f)\Bigr)
\,. \label{D}
\end{align}

\bigskip

In conclusion (cf. \eqref{eq:s1}, \eqref{eq:s2} and \eqref{D}) we
have shown that
\begin{equation*}
  \Var_\L(f)\le \frac{1}{1-\sqrt{\e_k}}\g_{k-1}\Bigl(\cD_\L(f)+\sum_{x\in
  \L_1\cap\L_2}\mu_\L\bigl(c_{x}^{\L}\Var_x(f)\bigr)\Bigr)\,.
\end{equation*}
Averaging over the $s_k$ possible choices of the sets $\L_1,\L_2$
gives (recall property (iii) at the beginning of the proof)
\begin{equation*}
  \Var_\L(f)\le \frac{1}{1-\sqrt{\e_k}}\g_{k-1}(1+\frac{1}{s_k})\cD_\L(f)
\end{equation*}
which implies that
\begin{align}
\g_k & \le \frac{1}{1-\sqrt{\e_k}}(1+\frac{1}{s_k})\g_{k-1}
\le
\g_{k_0}\ \prod_{j=k_0}^k  \frac{1}{1-\sqrt{\e_j}}(1+\frac{1}{s_j}) 
\g_{k_0-1}\ \prod_{j=k_0}^\infty \frac{1}{1-\sqrt{\e_j}} (
1+\frac{1}{s_j})\text{} \label{eq:east4}
\end{align}
where $k_0$ is the smallest integer such that $\d_{k_0}>1$.

\bigskip

By definition of the quantity $\e_k$ given in Proposition
\ref{gapblock} and by construction of the $\Lambda_{1,2}^{(i)}$'s,
 $|I| \geq \delta_k/8$, so that $\e_k \leq p^{\delta_k/8}$. The
convergence of the product in (\ref{eq:east4}) is thus guaranteed
and the positivity of the spectral gap follows.

Let us now discuss the asymptotic behavior of the gap as $q
\downarrow 0$. We first observe that $\g_{k_0-1} \leq \g_{k_0}<
(1/q)^{\a_\d}$ for some finite $\a_\d$. That follows {\it e.g.}\
from a coupling argument: in a time lag one and  with probability
larger than $q^{\a_\d}$ for suitable $\a_\d$, any configuration in
$\L_{k_0} \in \bbF_{k_0}$ can reach the empty configuration by just
flipping one after another the spins starting from the right
boundary. In other words, under the maximal coupling, two arbitrary
configurations will couple in a time lag one with probability larger
than $q^{\a_\d}$ \ie   $\g_{k_0}< (1/q)^{\a_\d}$. We now analyze the
infinite product (\ref{eq:east4}) which we rewrite as
$$
\prod_{j=k_0}^\infty\left(\frac{1}{1-\sqrt{\e_j}}\right)\,\prod_{j=k_0}^\infty\left(1+\frac{1}{s_j}\right).
$$
The second factor, due to the exponential growth of the scales, is
bounded by a constant independent of $q$.

To bound the first factor define
$j_*=\min\{j:\e_j\le \nep{-1}\}$ and observe that, for $q$ small enough
$$ \text{}
-2 + \frac{\log_2(1/q)}{1-3\d} \leq  j_* \leq  2 + \frac{\log_2(1/q)}{1-3\d} .
$$
Then write
\begin{gather} \label{eq:east1}
\prod_{j=k_0}^\infty\left(\frac{1}{1-\sqrt{\e_j}}\right)\le \prod_{j=1}^{j_*}\left(\frac{1+\sqrt{\e_j}}{1-\e_j}\right)
\prod_{j>j_*}^\infty\left(\frac{1}{1-\sqrt{\e_j}}\right)
\le \nep{C} \,2^{j_*}\,\prod_{j=1}^{j_*}\left(\frac{1}{1-\e_j}\right)
\end{gather}
where we used the bound $1/(1-\sqrt{\e_i})\le
1+\left(e/(e+1)\right)\sqrt{\e_j}$ valid for any $j\ge j_*$
together with
\begin{align*}
\sum_{j>j_*}^\infty\log\left(1+\frac{\nep{}}{\nep{}+1}\sqrt{\e_j}\right)
& \le
 \frac{\nep{}}{\nep{}+1}\sum_{j>j_*}^\infty\sqrt{\e_j} \\
& \le  \frac{\nep{}}{\nep{}+1} \int_{j_*-1}^\infty dx\
\exp(-q(2^{x(1-3\d)})/16) \text{}\\
& = A_\d\int_{2^{(j_*-1)(1-3\d)}}^\infty dz\ \exp(-q z/16)/z \\
& \le 16A_\d 2^{-(j_*-1)(1-3\d)}q^{-1}\exp(-q2^{(j_*-1)(1-3\d)}/16) \le C
\end{align*}
for some constant $C$ independent of $q$.

Observe now that $1-\e_j\ge 1-\nep{-q\d_j/8}\ge A q\d_j$ for any $j\le j_*$
and some constant $A$ independent of $q$. Thus the r.h.s.\ of \eqref{eq:east1}
is bounded from above by
\begin{gather*}
  e^C\,(\frac{2}{Aq})^{j_*}\,\prod_{j=1}^{j_*}\d_j^{-1}\le
  \frac{1}{q^a}\,(1/q)^{j_*}\,2^{-(1-3\d)j_*^2/2}
\leq \frac{1}{q^{a'}}\,(1/q)^{\log_2(1/q)/(2-6\d)}
\end{gather*}
for some constants $a, a'$ (independent of $q$).
This ends the proof of Theorem \ref{th:cmrt-gap}.
 \end{proof}

\begin{proof}[Proof of Theorem \ref{th:cmst-gap}]
Since it is always true that $\gap(\cL) \leq 1$ (take $f= \s_0$ in
\eqref{def_gap}) we can assume without loss of generality that $q$
is small. Then, thanks to Proposition \ref{prop:monotony},
$\gap(\cL) \leq \gap(\cL_{\L})$ with $\L=[0,\ell)$ and $\ell=1/q$
that we assume, for simplicity, to be an integer. In order to bound
from above $\gap(\cL_{\L})$, we will make use of the following
general result.

\begin{Lemma}\label{prop:amine}
For any $A \subset \Omega_\L$, the hitting time $\tau_A = \inf  \{ t \geq 0 : \sigma_t^\Lambda \in A\}$ satisfies
$$
\bbP^\L_{\pi_\L} (\tau_A >t) \leq e^{-t \gap(\cL_\L) \pi_\L(A)}\,.
$$
\end{Lemma}
\begin{proof}
It is well known (see e.g.\ \cite{A}) that $\bbP^\L_{\pi_\L} (\tau_A
>t) \leq e^{-t \l_A}$, where
$$
\lambda_A := \inf \left\{ \cD_\L(f) : \pi_\L(f^2)=1 \mbox{ and } f
\equiv 0 \mbox{ on } A \right\}\,.
$$
If $f \equiv 0$ on $A$   we can bound \begin{equation}
\begin{split} \Var _\L (f) & = \frac{1}{2} \sum _{\s \in \O_\L}
\sum _{\s' \in \O_\L} \p_\L(\s) \p_\L (\s') \bigl( f(\s)-f(\s')
\bigr)^2\\& \ge \sum_{\s \in A}\sum_{\s' \in \O_\L}\p_\L(\s) \p_\L
(\s')\bigl( f(\s)-f(\s') \bigr)^2= \p_\L(A) \p_\L (f^2 )\,.
\end{split}
\end{equation}
Hence, for such a function $f$ it holds
$$ \cD_\L(f) \ge \gap(\cL_\L) \Var _\L (f) \geq \gap(\cL_\L)\p_\L(A) \p_\L (f^2
)\,.$$ It  then follows  that $\lambda_A \geq \gap (\cL_\L)
\pi_\L(A)$.
\end{proof}

Denote by $\tau$ the first time there are $n:=\lfloor \log_2 \ell
\rfloor $ empty sites in $[0,\ell)$, and by $\tau_0$ the first time
there is
  an  empty site  at the origin. Thanks to Point
$(i)$ of Proposition \ref{prop:comb}   and since $2^n\leq \ell \leq
2^{n+1}-1$, starting from the filled configuration
 in order to
end up at time $\tau_0$ with an empty site located at the origin,
the system must have created before   $n+1\geq  n$ empty sites in
$[0,\ell)$. Hence $\tau \leq \tau_0$ when starting from
$\mathds{1}$. In turn, Lemma \ref{prop:amine} applied to the set
$A=\{\eta: \eta_0=0\}$ implies that
\begin{gather}\label{torta}
e^{-t \gap(\cL_\L) q}= e^{-t \gap(\cL_\L) \pi_\L(A)} \geq
\bbP^\L_{\pi_\L} (\tau_0 >t) \geq \pi_L(\mathds{1} )
\bbP^\L_{\mathds{1}} (\tau
>t)\,.
\end{gather}
Recall the definition of $V(n)$ introduced in Section \ref{sec:comb}
and denote by $\Omega_n \subset V(n)$ those configurations with
exactly $n$ zeros. Thanks to Point $(ii)$ of Proposition
\ref{prop:comb}, for a suitable constant $c \leq 0.7$ it holds
\begin{equation}\label{piccolini}
\begin{split} \pi_L (\mathds{1} ) \bbP^\L_{\mathds{1} } (\tau \leq t)
& \leq
\bbP^\L_{\pi_\L} (\tau \leq t) \leq \bbP^\L _{\pi_\L} (\exists s \leq
t\,:\, \eta_s \in \O_n) 
  \leq  (t/q) \pi_\L(\O_n) \\
  & \leq  (t/q) q^{n}  |\O_n|\leq (t/q) q^n
2^{\genfrac{(}{)}{0pt}{}{n}{2}} n!c^n\leq t q^{(n/2)(1+o(1)) }\,.
\end{split}
\end{equation}
 Here $o(1)$ tends to zero as $n$
goes to infinity and therefore when $q\downarrow 0$. To prove the
third inequality in \eqref{piccolini}, observe that we can write
$\eta_t$ as time change of a discrete time Markov chain. More
precisely, it holds  $ \eta_t= \eta^{d.t.}_{N_t}$ where
$\eta^{d.t.}$ is the discrete time Markov chain on $\O_\L$ whose
transition matrix $\bbP$ satisfies $\bbI- \bbP= |\L|^{-1}\cL _\L$,
and where $(N_t)_{t\geq 0}$ represents a Poisson process with mean $\bbE(
N(t))= |\L| t$. Trivially, $\pi_\L$ is reversible for $\eta^{d.t.}$.
We write $P$ for the law of $\eta ^{d.t.}$ having $\pi_\L$ as
initial distribution. Then 
\begin{align*}
\bbP^\L _{\pi_\L} (\exists s \leq t\,:\, \eta_s \in \O_n)
& =  
\bbP^\L _{\pi_\L} (\exists s < t\,:\, \eta_s \in \O_n)\\
& =  
\sum _{k=0}^\infty \bbP( N(t)=k) P_{\pi_\L}( \exists j: 0\leq j < k \text{ and }\eta^{d.t.}_j\in \O_n)\\ 
& \leq
\sum _{k=0}^\infty \bbP( N(t)=k) k \pi_\L(\O_n) = \bbE(
N(t) ) \pi_\L(\O_n)= (t/q) \pi_\L (\O_n)\,.
\end{align*}
 The above
inequality \eqref{piccolini} together with \eqref{torta} implies
that
\begin{equation}\label{tortabis}
e^{-t \gap(\cL_\L) q} \geq \pi_\L (\mathds{1}) -t q^{(n/2)(1+o(1))
}\,. \end{equation} Since  $\pi_L(\mathds{1})=(1-q)^{1/q}\sim 1$ ,
we choose $t=q^{-(n/2)(1-\frac{\delta}{2})}$ so that the above
r.h.s.\ is at least $1/2$ for $q$ small enough.  
Finally, we conclude from
\eqref{tortabis} that for $q$ small, $ t\gap(\cL_\L) q \leq \log 2$.
The expected result follows.
\end{proof}

We end this section with the proof of Proposition \ref{gapblock}.

\begin{proof}[Proof of Proposition \ref{gapblock}]
For any mean zero function $f\in L^2(\O_\L,\pi_\L)$ let
$$
\pi_1f:=\pi_{B_2}(f), \quad \pi_2f:=\pi_{B_1}(f)
$$
be the natural projections onto $L^2(\O_{B_i},\pi_{B_i})$,
$i=1,2$. Obviously $\pi_1\pi_2f=\pi_2\pi_1f=0$. The generator of the block dynamics can then be
written as:
$$
\cL_{\rm block}f=c_1\bigl(\pi_2f -f \bigr) +\pi_1f -f
$$
and the associated eigenvalue equation as
\begin{equation}
  \label{eq:4bis}
c_1\bigl(\pi_2f -f \bigr) +\pi_1f -f=\l f.
\end{equation}
By taking $f(\s_\L)=g(\s_{B_2})$  with $g$ non zero and with $\p_1
g=0$, we see that $\l=-1$ is an eigenvalue. Moreover, since $c_1\le
1$, $\l\ge -1$. Assume now $0>\l> -1$ and apply $\pi_2$ to both
sides of \eqref{eq:4bis} to obtain (recall that $c_1=c_1(\s_{B_2})$)
\begin{equation}
  \label{eq:5}
  -\pi_2f=\l\pi_2f \quad \imp \quad \pi_2f=0
\end{equation}
For any $f$ with $\pi_2f=0$ the eigenvalue equation becomes
\begin{equation}
  \label{eq:6}
  f=\frac{\pi_1 f}{1+\l+c_1}
\end{equation}
and that is possible only if
$$
1=\pi_1(\frac{1}{1+\l+c_1})= \frac{1}{1+\l} \p_{B_2}(c_1=0) +
\frac{1}{2+\l} \p_{B_2}(c_1=1)\,.
$$
We can  solve the equation to get
$$
\l=-1+\sqrt{1-\pi_{B_2}(c_1)}\le -1+\sqrt{\e_k}\,.
$$\qedhere \end{proof}


\section{Persistence function: proof of Theorem \ref{th:persistence}} \label{sec:persistence}

We follow \cite{CMRT}. Observe first that $F(t)=F_1(t)+F_0(t)$ where
$$
F_1(t)= \int\, d\pi(\h)\, \bbP_\h (\s_0(s)=1\ \text{for all $s\le
t$})
$$
and similarly for $F_0(t)$. We will prove the exponential decay of
$F_1(t)$, the case of $F_0(t)$ being similar.

For any $\l>0$ the exponential Chebychev inequality gives
$$
F_1(t)= \int\, d\pi(\h)\, \bbP_\h \Bigl(\int_0^t ds\,
\s_0(s)=t\Bigr) \le \nep{-\l t}\ \bbE_\p\bigl(\nep{\l\int_0^t ds
\,\s_0(s)}\bigr)
$$
where we recall that $\bbE_\pi$ denotes the expectation over the
process started from the equilibrium distribution $\pi$. Consider
the self-adjoint operator $H_\l:=\cL + \l V$, on $L^2(\pi)$, where
$V$ is the multiplication operator by $\s_0$. By the very definition
of the scalar product $\scalprod{f}{g}$ in $L^2(\pi)$ and the
Feynman-Kac formula, we can rewrite $\bbE_\pi(\nep{\l\int_0^t
\s_0(s)})$ as $\scalprod{{\bf 1}}{\nep{tH_\l}{\bf 1}}$. Thus, if
$\b_\l$ denotes the supremum of the spectrum of $H_\l$,
$\bbE_\pi(\nep{\l\int_0^t \s_0(s)})\leq \nep{t\b_\l}$. In turn,
$$
F_1(t) \leq e^{-\lambda t (1-\frac{\beta_\lambda}{\lambda})} .
$$
Hence, in order to complete the proof we need to show that for suitable
positive $\l$ the constant $\b_\l/\l$ is strictly smaller than one.

For any function $f$, with $\| f\|_{\mathbb{L}^2(\pi)}=1$, in the domain of $H_\l$ (which coincides
with Dom($\cL$)), write $f=\a {\bf 1} +g$ with
$\scalprod{{\bf 1}}{g}=0$. Thus, by Cauchy-Schwarz inequality and the fact that $\mathcal{L}{\bf 1} =0$ and that  $|\sigma_0| \leq 1$, we have
\begin{align*}
\scalprod{f}{H_\l f}
& =
\scalprod{g}{\cL g}  + \a^2 \l\scalprod{{\bf 1}}{V {\bf 1}}  + \l\scalprod{g}{Vg} + 2\l\a\scalprod{{\bf 1}}{Vg}  \\
& \le (\l-\gap(\cL)) \scalprod{g}{g} +\a^2\l p
+2\l|\a|\bigl(\scalprod{g}{g}pq\bigr)^{1/2} .\text{}
\end{align*}
 Since
$\| f\|_{\mathbb{L}^2(\pi)}=1$, $\a^2+\scalprod{g}{g}=1$ and
\begin{equation}
  \frac{\b_\l}{\l} \le \sup_{0\le \a\le 1}\Bigl\{(1-\frac{\gap(\cL)}{\l})(1-\a^2)+p\a^2+2\a\bigl((1-\a^2)pq\bigr)^{1/2}\Bigr\}
\label{eq:pers2}
\end{equation}
If we choose $\l=\gap(\cL)/2$ the r.h.s.\ of \eqref{eq:pers2}
becomes
\begin{align*}
& \sup_{0\le \a\le 1} \left\{ (1+p)\a^2 -1 + 2\a\bigl((1-\a^2)pq\bigr)^{1/2} \right\} \\
& \qquad \qquad \le \sup_{0\le \a \le 1} \left\{(1+p)\a^2 -1 +
2\bigl((1-\a^2)pq\bigr)^{1/2} \right\}= \frac{pq}{1+p}+p <1
\end{align*}
since $p\neq 1$.
Thus $F_1(t)$ satisfies
$$
F_1(t)\le \nep{-t\frac{\gap(\cL)}{2} \frac{q}{1+p}}.
$$
A very similar computation (details are left to the reader) shows that
$$
F_0(t)\le \nep{-t\frac{\gap(\cL)}{2} \frac{p}{1+p}} .
$$
This ends the proof of Theorem \ref{th:persistence}.


\section{Log-Sobolev constant: proof of Theorem \ref{th:ls}} \label{sec:ls}

The aim of this section is to prove Theorem \ref{th:ls}.

Fix a finite interval $\Lambda$ and assume for simplicity that $\Lambda=[1,L-1]$.

\medskip

We start with the easy part, namely the upper bound on $C_\alpha(\Lambda)$.
We observe first that, thanks to \cite[Theorem 1.8]{mossel}, the $\alpha$-log-Sobolev inequality with constant
$C_\alpha(\Lambda)$ implies any $\beta$-log-Sobolev inequality with the same constant
$C_\alpha(\Lambda)$ as soon as $\beta \leq \alpha$. Hence, we have
$C_\alpha(\Lambda) \leq C_2(\Lambda)$ for any $\alpha \in [0,2]$ so that we only have to prove that
$C_2(\Lambda) \leq cL$. This in turn is implied by the following well-known property of Diaconis and
Saloff-Coste \cite[Corollary A.4]{diaconis-saloff-coste}:
$$
C_2(\Lambda) \leq \frac{1}{\gap(\cL_\Lambda)}\frac{\log\frac{1}{\pi_\Lambda^*} -1}{1-2\pi_\Lambda^*}
$$
where $\pi_\Lambda^*:=\min_{\sigma \in \Omega_\Lambda} \pi_\Lambda(\sigma)=\min(p,q)^{L-1}$.
The expected upper bound follows at once from Theorem \ref{th:gap}.

\medskip

Now we turn to the lower bound on $C_\alpha(\Lambda)$. The result will be achieved by a test function. Define the
random variable $\xi \in \{1,\dots,L\}$ as the the distance from the leftmost empty site in  $\Lambda$ to the left boundary of $\Lambda$, namely
(recall that $\Lambda=[1,L-1]$),
$$
\xi (\sigma) := \min_{x \in \Lambda: \sigma(x)=0} \{ d(x,0) \}
$$
with the convention that $\min \emptyset = L$. In other words, $\xi$ is the position, in $\{1,\dots,L\}$, of the leftmost empty site (including the boundary condition at site $L$).
Then, for $g \colon \{1,\dots,L\} \to \bbR$, let $f\colon \Omega_\Lambda \to \bbR$
defined as $f(\sigma)=g(\xi(\sigma))$. Define also the distribution $m$ on
the set $\{1,\dots,L\}$ by  $m(k)=\pi_\Lambda(\{\xi=k\})$ so that
$$
m(k) =
\begin{cases}
q p^{k-1} & \mbox{if } k=1,2\dots, L-1, \\
p^{L-1} & \mbox{if } k=L .
\end{cases}
$$
Hence, the $\alpha$-log-Sobolev inequality \eqref{ls}, applied to $f$, reads
\begin{equation}\label{eq:bd}
m(g \log (g/m(g))) = \ent_{\Lambda}(f)
\leq
\frac{\alpha \alpha' C_\alpha(\Lambda)}{4}
\mathcal{D}_\Lambda (g(\xi)^{1/\alpha},g(\xi)^{1/\alpha'}) .
\end{equation}
Let us analyze the right hand side of the latter.
By definition of the Dirichlet form, we have
\begin{align*}
\mathcal{D}_\Lambda (g(\xi)^{1/\alpha},g(\xi)^{1/\alpha'}) & =
\sum_{k=1}^{L}  \sum_{\sigma : \xi(\sigma) = k} \pi_\Lambda (\sigma)
\sum_{x=1}^{L}  c_x^\Lambda (\sigma) [(1-\sigma(x))p + \sigma(x)q] \times \\
& \phantom{AAA} \times (g(\xi(\sigma^x))^{1/\alpha} - g(\xi(\sigma))^{1/\alpha})
(g(\xi(\sigma^x))^{1/\alpha'} - g(\xi(\sigma))^{1/\alpha'}) \\
& = m(L-1) p (g(L)^{1/\alpha} - g(L-1)^{1/\alpha})(g(L)^{1/\alpha'} - g(L-1)^{1/\alpha'}) \\
& \quad
+\sum_{k=1}^{L-2} m(k) pq
(g(k+1)^{1/\alpha} - g(k)^{1/\alpha})(g(k+1)^{1/\alpha'} - g(k)^{1/\alpha'}) \\
& \quad
+\sum_{k=2}^{L} m(k) q (g(k-1)^{1/\alpha} - g(k)^{1/\alpha})(g(k-1)^{1/\alpha'} - g(k)^{1/\alpha'}),
\end{align*}
where we used that, given that $\xi=k$, the only possible flips are at site $x=k-1$ (for $k \geq 2$), and at site $x=k$, in which case the flip is admissible only if site  $x+1$ is empty (hence the extra factor $q$), except for $\xi=L-1$ where by definition the constraint is always satisfied, due to the boundary condition at site $x=L$. By a change of variable, and using that
$m(k)pq=m(k+1)q$, we arrive at
\begin{align*}
\mathcal{D}_\Lambda (g(\xi)^{1/\alpha}, & g(\xi)^{1/\alpha'})
 =
2 \sum_{k=1}^{L-2} m(k) pq
(g(k+1)^{1/\alpha} - g(k)^{1/\alpha})(g(k+1)^{1/\alpha'} - g(k)^{1/\alpha'}) \\
& + m(L-1) p (1+q) (g(L)^{1/\alpha} - g(L-1)^{1/\alpha})(g(L)^{1/\alpha'} - g(L-1)^{1/\alpha'}) . \\
\end{align*}
Observe that the latter corresponds to the Dirichlet form associated to the birth and death process on $\{1,\dots,L\}$
with reversible measure $m$ and transition rates $p(k,k+1)=2pq$ for $k=1,\dots,L-2$, and $p(L-1,L)=p(1+q)$
(and $p(k+1,k)$ computed so that the detailed balanced condition $m(k)p(k,k+1)=m(k+1)p(k+1,k)$ holds).
In turn, Inequality \eqref{eq:bd} is nothing but the $\alpha$-Sobolev inequality corresponding to this birth and death process
\footnote{In fact it is possible to prove that the $\alpha$-Sobolev constant associated to this birth and death process compares to
$L$ for any $\alpha \in (0,2]$ and that the Poincar\'e inequality ({\it i.e.}\ the case $\alpha=0$) holds with a constant $C_0$ independent of $L$.}.

Consider now the special choice $g(k)= \lambda^{k-1}$, defined on $\{1,\dots,L\}$, with $\lambda \in (0,1/p)$ a parameter that will be chosen later. The above Dirichlet form reduces, after simple algebra, to
\begin{align*}
\mathcal{D}_\Lambda (g(\xi)^{1/\alpha},  g(\xi)^{1/\alpha'})
& = (\lambda^\frac{1}{\alpha}-1)(\lambda^\frac{1}{\alpha'}-1)\left( q(1+q)(p \lambda)^{L-1}
 +  2q^2 \sum_{k=1}^{L-2} (p\lambda)^k \right) \\
& =
(\lambda^\frac{1}{\alpha}-1)(\lambda^\frac{1}{\alpha'}-1) \frac{pq}{1-p\lambda}
\left[p(1-(1+q)\lambda)(p\lambda)^{L-2}  + 2q \right] .
\end{align*}
Denote by $\e = 1-p\lambda \in (0,1)$. By a Taylor expansion, as $\e$ goes to 0, we thus have
\begin{equation} \label{epsilon1}
\mathcal{D}_\Lambda (g(\xi)^{1/\alpha},  g(\xi)^{1/\alpha'})
=
\left[\frac{1}{p^\frac{1}{\alpha}}-1\right]\left[\frac{1}{p^\frac{1}{\alpha'}}-1\right]
pq(2qL + 1 +q +o(1)) .
\end{equation}
On the other hand,
\begin{align*} 
m(g)
& =
\sum_{k=1}^L m(k) g(k) = (p \lambda)^{L-1} + q  \sum_{k=1}^{L-1} (p \lambda)^{k-1}
=
\frac{(p \lambda)^{L-1} p(1-\lambda) + q}{1-p\lambda} \\
& =
1 + q(L-1) + o(1) \nonumber.
\end{align*}
Hence,
$$
\frac{d}{d \log \lambda}m(g) = \frac{pq\lambda + (p \lambda)^{L-1}[p(1-\lambda)(1-p\lambda)(L-1)-pq\lambda]}{(1-p\lambda)^2} .
$$
Since $\frac{d}{d \log \lambda}m(g) = \sum_{k=1}^L m(k)kg(k)$, we deduce that
\begin{align*}
m(g \log g)
& =
\log \lambda \sum_{k=1}^L m(k)kg(k)
=
\log \lambda \frac{pq\lambda + (p \lambda)^{L-1}[p(1-\lambda)(1-p\lambda)(L-1)-pq\lambda]}{(1-p\lambda)^2} \\
& =
\frac{1}{2}(L-1)(qL+2p)\log(1/p) + o(1).
\end{align*}
In turn, there exists a constant $c$, that may depend on $q$ but that is independent of $L$ and $\alpha$, such that
\begin{align*}
m(g \log (g/m(g))) = m( g \log g) - m(g) \log m(g) \geq cL^2 + o(1) .
\end{align*}
The expected result finally follows, using \eqref{epsilon1}, from \eqref{eq:bd}, in the limit $\e \to 0$.
This achieves the proof of the lower bound on the constants $C_\alpha(\Lambda)$ and therefore of Theorem \ref{th:ls}.


\section{Out of equilibrium I, long time behavior: proof of Theorem \ref{th:outI}} \label{sec:outI}

The aim of this section is to prove Theorem \ref{th:outI}. To  that
purpose, we will deeply use the oriented character of the East
process. We need some preparation.

In \cite{Aldous}, Aldous and Diaconis introduced the following notion of \emph{distinguished zero}:

\begin{definition}[Distinguished zero \cite{Aldous}] \label{def:distinguished}
Fix an initial configuration $\sigma \in \Omega$, suppose that
$\s(x)=0$ and call the site $x$ distinguished.  Then, setting
$\xi_0(\s)=x$,  the position $\xi_s=\xi_s(\sigma) \in \mathbb{Z}$ of
the distinguished zero at time $s >0$ obeys the following iterative
rule.  $\xi_s=x$ for all times $s$ strictly smaller than the first
legal ring (recall the graphical construction of Section
\ref{graphical}) of the mean one Poisson clock associated to site
$x$ when it jumps to site $x+1$. Then it waits the next legal ring
at $x+1$ and when this occurs it jumps to $x+2$, and so on.
\end{definition}

Thus, with probability one, the path $\{ \xi_s \}_{s \in [0,t]}$ is right-continuous, piecewise
constant, non decreasing, with at most a finite number of discontinuities at which it increases by (exactly) one.
 See Figure \ref{fig:distinguished} where $t_1, \dots ,t_4$ are legal rings. Also note that, by definition of the legal rings, necessarily $\sigma_s(\xi_s)=0$ for all $s$ (hence the name distinguished zero).

\begin{remark} \label{berlufuori}
One important feature of the distinguished zero is the following property that we will often use in the sequel.
Fix a starting configuration $\sigma$ with $\sigma(b)=0$. Make $b$ distinguished. Then, given the path $\{\xi_s\}_{s \in [0,t]}$,
 the law of ${\sigma_t}_{[a,\xi_t)}$  (\ie the restriction of $\s_t$ to $[a,\xi_t) $) depends only on $\sigma_{[a,b)}$
and not on $\sigma_{[a,b)^c}$.
\end{remark}

In the following we use the standard  notation $f(u-)=\lim_{\e \downarrow 0} f(u-\e)$.

By exploiting the oriented character of the East process, and more precisely the fact that the motion of the distinguished zero for
$s>t$ cannot be influenced by the clock rings and coin tosses in $(-\infty,\xi_t)$, Aldous and Diaconis proved the following important fact: if one starts with the equilibrium measure $\pi$, say on $(-\infty,x)$, with $x$ distinguished, then the process is still
at equilibrium, at any time, on $(-\infty,\xi_t)$. More precisely:

\begin{Lemma}[\cite{Aldous}] \label{lem:distinguished}
 Consider the East process on $\bbZ$. Fix an interval $\Lambda
= [a,b)$ with possibly $a=-\infty$ (in which case
$\Lambda=(-\infty,b)$). Assume that, at time zero,  $\sigma(b)=0$
while $\sigma_\Lambda$ is distributed according to the equilibrium
measure $\pi_\Lambda$. Make $b$ distinguished and call $\xi_s$ its
position at time $s$. Then, the conditional distribution of
${\sigma_t}_{[a,\xi_t)}$ ({\it i.e.}\ $\sigma_t$ restricted to
$[a,\xi_t)$) given the path $\{ \xi_s\}_{s \leq t}$ is the
equilibrium measure $\pi_{[a,\xi_t)}$.
\end{Lemma}

\begin{figure}[ht]
\psfrag{a}{ $a$}
 \psfrag{b}{$b$}
 \psfrag{t1}{$t_1$}
 \psfrag{t2}{$t_2$}
 \psfrag{t3}{$t_3$}
 \psfrag{t4}{$t_4$}
 \psfrag{t0}{$t=0$}
 \psfrag{t}{$t$}
 \psfrag{z}{$\bbZ$}
 \psfrag{xit}{$\xi_t$}
 \psfrag{pi}{$\pi_\Lambda$}
 \psfrag{pit}{$\pi_{[a,\xi_t)}$}
 \includegraphics[width=.80\columnwidth]{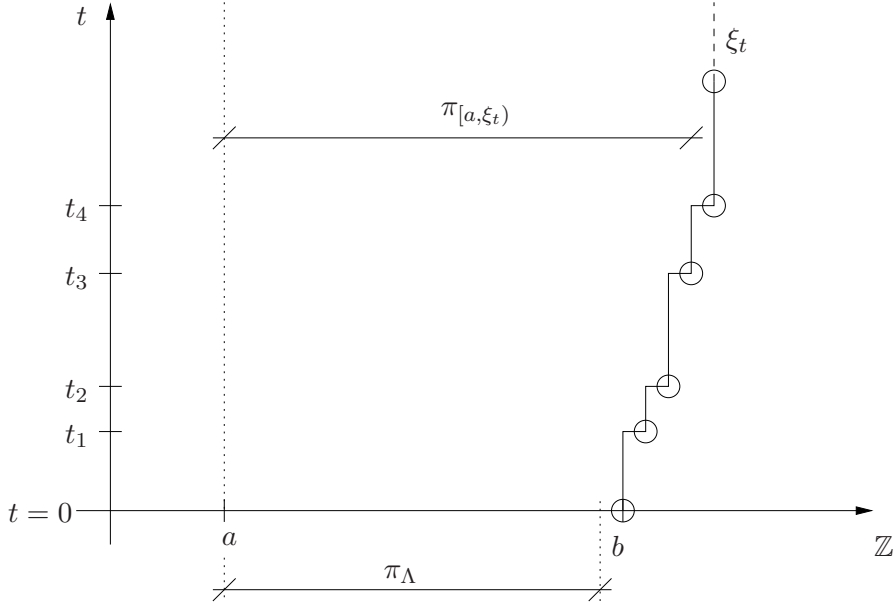}
\caption{The path $\{\xi_s\}_{s \geq 0}$ of the distinguished zero
$b$.  Illustration of Lemma \ref{lem:distinguished} on the interval
$\Lambda=[a,b)$.} \label{fig:distinguished}
 \end{figure}

We prove Lemma \ref{lem:distinguished} for completeness.

\begin{proof}[Proof of Lemma \ref{lem:distinguished}]
Let $0 < t_1 < t_2 < \dots$ be the (random) times when the
distinguished zero jumps, {\it i.e.}\ increases by one. Setting
$t_0=0$,   at time $s \in [t_i,t_{i+1})$ the position of the
distinguished zero is precisely $\xi_s = b+i$ (see Figure
\ref{fig:distinguished}).  The proof goes by induction on the
interval $[t_i, t_{i+1})$ containing $t$. The case $i=0$ follows
from the stationarity of $\p_\L$ for the East process on $\L$ with a
frozen zero on the right. Assume that for all $t \in [t_{i-1}, t_i)$
the law of ${\sigma_t}_{[a,\xi_{t_i-})}$, given the path  $\{
\xi_s\}_{s \leq t}$, is the equilibrium measure
$\pi_{[a,\xi_{t_i-})}$. Given the legal ring at site $b+i-1$ at time
$t_i$, we conclude that the law  of
 ${\sigma_{t_i}}_{[a,\xi_{t_i})}$ is the equilibrium measure $\pi_{[a,\xi_{t_i})}$:
indeed, the new configuration ${\sigma_{t_i}}_{[a,\xi_{t_i})}$
consists of the existing configuration ${\sigma_{t_i-}
}_{[a,\xi_{t_i-})}$ (distributed according to $\pi_{[a, \xi_{t_i-}
)} =\pi_{[a,b+i-1)}$) together with the new configuration at site
$b+i-1$ created by the legal ring, which is also Bernoulli$(1-q)$,
thus making the whole configuration distributed according to $
\p_{[a, \xi_{t_i})}=\pi_{[a,b+i)}$. Then, at any subsequent time
$t\in [t_i, t_{i+1})$, knowing that there is no legal ring at site
$b+i$ and
 by stationarity,  the law  of
${\sigma_t}_{[a,\xi_{t})}={\sigma_t}_{[a,\xi_{t_i})}$ remains
i.i.d.\ Bernoulli$(1-q)$.
This carries the induction  forward and ends the proof of the lemma.
\end{proof}

We are now  in position to prove Theorem \ref{th:outI}.

\begin{proof}[Proof of Theorem \ref{th:outI}]
Let $f$ be a local function, and assume that its support is included in $[a,a']$. Assume for simplicity that $\pi(f)=0$.
Given a configuration $\sigma$, let $b=b(\sigma)=\inf \{x \geq a'+1 \mbox{ s.t. } \sigma(x)=0\}$ be the position of the first empty site
 in $\sigma$ on the right of $a'$ (it exists $Q$ a.s.). Make $b$ distinguished and denote by $\xi_s$ its position at time $s$.
Given the path $\{\xi_s\}_{s \leq t}$, let $0 < t_1 < t_2 < \dots< t_{n-1}<t$ be the times when the distinguished zero jumps,
and set $t_0=0$, $t_n=t$. By construction, $\xi_s=b+i$ for any $s \in [t_i,t_{i+1})$, see Figure \ref{fig:distinguished}.

 Since the support of $f$ is included in $[a,a'] \subset [a,\xi_t)$,
thanks to Lemma \ref{lem:distinguished} it holds
\begin{align} \label{berluingalera}
\pi_{[a,b)} \left( \bbE_.(f(\sigma_t) \tc \{\xi_s\}_{s \leq t}) \right)
 =
\int d\pi_{[a,b)}(\sigma) \bbE_\sigma(f({\sigma_t}_{[a,\xi_t)}) \tc \{\xi_s\}_{s \leq t})
 =
\pi_{[a,\xi_t)} (f)
 =0 .
\end{align}
Note that the notation used in the above first member is justified
by Remark \ref{berlufuori}. The same remark will be frequently used
below in our notational choice.

For any $\sigma$, thanks to \eqref{berluingalera} and the
Cauchy-Schwarz inequality,
\begin{align}  \label{mortesua}
|\bbE_\sigma\bigl(f(\sigma_t)\bigr)| & \leq
\bbE_\sigma \left( |\bbE_\sigma(f(\sigma_t) \tc \{\xi_s\}_{s \leq t}) | \right) \nonumber \\
& \leq
\frac{1}{ (p \wedge q)^{b-a}} \bbE_\sigma \left( \int d\pi_{[a,b)}(\eta) |\bbE_\eta(f(\eta_t) \tc \{\xi_s\}_{s \leq t}) | \right) \nonumber \\
& \leq
\frac{1}{ (p \wedge q)^{b-a}} \bbE_\sigma \left( \var_{\pi_{[a,b)}} (\bbE_.(f(\eta_t) \tc \{\xi_s\}_{s \leq t}))^{1/2} \right) .
\end{align}
Now our aim is to control the right hand side of the latter, using
the Poincar\'e Inequality. Denote by $V_i:=[a,b+i)$ for
$i=0,1\dots,n-1$ (that corresponds to $[a,\xi_s)$ when $s \in
[t_i,t_{i+1})$) and by $\{P_s^{(i)}\}_{s \in [t_i,t_{i+1})}$ the
Markov semigroup associated to the East process in the interval
$V_i$ with a fixed zero boundary condition at site $b+i$. Given the
path $\{\xi_s\}_{s \leq t}$, thanks to Remark \ref{berlufuori},
${\sigma_t}_{[0,\xi_t)}$ coincides with the process obtained from
the initial configuration $\sigma_{V_0}$ evolving according to
$\{P_s^{(0)}\}_{s \in [t_0,t_{1})}$, up to time $t_1$, then evolving
according to $\{P_s^{(1)}\}_{s \in [t_1,t_{2})}$, up to time $t_2$,
and so on. Hence, if one writes for simplicity $\sigma \otimes
\sigma'=\sigma_{V_0} \sigma'_{\{b\}}$ for the configuration in
$\{0,1\}^{V_1}$ equal to $\sigma$ on $V_0$ and to $\sigma'$ on
$\{b\}$, then, for any $\eta$,
\begin{align*}
\bbE_{\eta}(f(\eta_t) \tc \{\xi_s\}_{s \leq t})
& =
\bbE_{\eta_{V_0}}(f(\eta_t) \tc \{\xi_s\}_{s \leq t}) \\
& =
\sum_{\genfrac{}{}{0pt}{}{\sigma' \in \{0,1\}}{\sigma \in \{0,1\}^{V_0}}}
 P_{t_1}^{(0)}(\eta_{V_0},\sigma)\pi_b(\sigma')
\bbE_{\sigma \otimes \sigma'}(f((\sigma \otimes \sigma')_{t-t_1}) \tc \{\xi_s\}_{s \in[t_1,t]})
\end{align*}
Therefore,
\begin{align*}
& \var_{\pi_{V_0}} \left(\bbE_{.}(f(\eta_t) \tc \{\xi_s\}_{s \leq t})  \right) \\
& \qquad \qquad \leq e^{-2 \gap(\cL_{V_0})t_1} \var_{\pi_{V_0}}
\left(
\sum_{\sigma' \in \{0,1\}} \pi_b(\sigma') \bbE_{\sigma \otimes \sigma'}(f((\sigma \otimes \sigma')_{t-t_1}) \tc \{\xi_s\}_{s \in[t_1,t]}) \right) \\
& \qquad \qquad \leq
e^{-2 \gap(\cL)t_1} \var_{\pi_{V_1}} \left(
 \bbE_{.}(f(\eta_{t-t_1}) \tc \{\xi_s\}_{s \in[t_1,t]}) \right)
\end{align*}
where we used Proposition \ref{prop:monotony} to bound from below
$\gap(\cL_{V_0})$ by $\gap(\cL)$, and the convexity of the variance.
The same procedure leads to
\begin{align*}
& \var_{\pi_{V_1}} \left( \bbE_{.}(f(\eta_{t-t_1}) \tc \{\xi_s\}_{s \in[t_1,t]}) \right)  \\
& \qquad \qquad \leq e^{- 2 \gap(\cL)(t_2-t_1)}
\var_{\pi_{V_2}} \left(
 \bbE_{.}(f((\eta)_{t-t_1-t_2}) \tc \{\xi_s\}_{s \in[t_2,t]})
 \right)
\end{align*}
so that, by a simple induction, we get
$$
\var_{\pi_{V_0}} \left(\bbE_{.}(f(\eta_t) \tc \{\xi_s\}_{s \leq t})
\right) \leq e^{-2 \gap(\cL) t} \var_{\pi_{[a, \xi_t]} } (f) .
$$
Plugging this bound into \eqref{mortesua} leads to
$$
|\bbE_\sigma(f(\sigma_t)| \leq \frac{e^{- \gap(\cL) t}}{ (p \wedge
q)^{b-a}} \bbE_\sigma \left( \var_{\pi_{[a, \xi_t]}} (f)^{1/2}
\right) \leq \frac{e^{- \gap(\cL) t}}{ (p \wedge q)^{b-a}} \|
f\|_\infty.
$$
Fix $\delta >0$. From the latter, we finally get
\begin{align*}
 \int dQ(\sigma) |\bbE_\sigma(f(\sigma_t)) - \pi(f)|
& =
\int dQ(\sigma) |\bbE_\sigma(f(\sigma_t))| \mathds{1}_{b \leq a' + \delta t} \\
& \quad + \int dQ(\sigma) |\bbE_\sigma(f(\sigma_t))| \mathds{1}_{b > a'+ \delta t} \\
& \leq
\left( e^{- \gap(\cL) t} \int  \frac{dQ(\sigma)}{ (p \wedge q)^{b(\sigma)-a+\delta t}} + Q(b > a'+ \delta t) \right) \| f\|_\infty  \\
& = \left( e^{- \gap(\cL) t} \frac{1-\alpha}{ (p \wedge q)^{a'-a}}
\sum_{k=0}^{\delta t} \left( \frac{\alpha}{ p\wedge q} \right)^k +
\alpha^{\delta t +1} \right) \| f\|_\infty  .
\end{align*}
Now we distinguish between two cases. $(i)$ If $\alpha < p \wedge
q$, then one lets $\delta$ tend to $+\infty$ so that the expected
result immediately follows. $(ii)$ if $\alpha > p \wedge q$, then
one chooses $\delta = \frac{\gap(\cL)}{2 \log(\alpha/q)}$. The
expected result follows after some simple algebra and few
simplifications left to the reader. This ends the proof.
\end{proof}


\section{Out of equilibrium II, aging and plateau behavior: proof of Theorem \ref{plateau}} \label{sec:outII}

The aim of this section is to prove Theorem \ref{plateau}. We will
not give the complete proof (that is quite long and involved). The
interested reader may  however found it in \cite{FMRT-cmp} under the
additional condition on $\mu$ that  $\mu([k,\infty))>0$ for any $k
\in \bbN$. Here, we shall only explain briefly the extra ingredient
that we use in order to remove this  condition.
 To this aim,  we will  use   different technical lemmas from \cite{FMRT-cmp} that we recall, for
completeness, at the end of this section.

\begin{proof}
The technical condition on $\mu$ is used in Section 4.2 of
\cite{FMRT-cmp}, namely on the finite volume approximation. More
precisely, it is used to guarantee the existence of infinitely many
empty sites  that remain empty up to a final fixed time $t_N$ with
$N$ fixed (recall Definition \ref{stalling-active}). Such empty
sites then allow to compare the East process on $[0,\infty)$ with
the East process on a finite box $[0,L]$.

The strategy, in order to remove the technical condition on $\mu$, that we adopt here is the following.
 We may prove that the process itself creates infinitely many empty sites that remain empty up to the final fixed time $t_N$. Given this, the proof remains unchanged with respect to \cite{FMRT-cmp}. In turn, this is a consequence of Lemma \ref{chepalle} below. Hence one only needs to prove Lemma \ref{chepalle}.

Recall that $d \in \bbN\setminus\{0\}$ is the smallest length such that
$\mu(\{d\})>0$, and that $n_d$ is the smallest integer $n$ such that  $d \in [2^{n-1}+1,2^n]$ (we can
assume $d>1$ otherwise the setting is the same of \cite{FMRT-cmp}). In the sequel, we fix $n=n_d$, for simplicity of notations.

Observe first that, since $Q$ is renewal, for any $j \geq 1$, almost
surely there exist infinitely many sites $\{x^{(m)}\}_{m \in \bbN}$
such that, $\sigma(x^{(m)}+id)=0$ for $i=0,\dots,j$, and
$\sigma(y)=1$ for all $y \in
[x^{(m)},x^{(m)}+jd]\setminus\{x^{(m)}+id, i=0,\dots j\}$, {\it
i.e.}\ infinitely many collections of $j+1$ consecutive empty sites
at distance $d$ one from the next one.  Denote for simplicity by
$x_1$ the first positive site satisfying the above property and set
$x_i=x_1+(i-1)d$, $i=2,\dots,j+1$. Denote by
$\Lambda_i=[x_i+1,x_{i+1}-1]$, $i=1,\dots,j$ (See Figure \ref{out}).

\begin{figure}[ht]
\psfrag{d}{ $d$}
 \psfrag{}{\footnotesize $\epsilon$}
 \psfrag{x1}{\footnotesize $x_1$}
 \psfrag{x2}{\footnotesize $x_2$}
 \psfrag{x3}{\footnotesize$x_3$}
 \psfrag{xj-1}{\footnotesize$x_{j-1}$}
 \psfrag{xj}{\footnotesize$x_j$}
 \psfrag{xj+1}{\footnotesize$x_{j+1}$}
\psfrag{l1}{$\!\!\! \Lambda_1$}
\psfrag{l2}{$\!\!\!\!\Lambda_2$}
\psfrag{lj-1}{$\!\!\Lambda_{j-1}$}
\psfrag{lj}{$\!\!\Lambda_{j}$}
 \includegraphics[width=.80\columnwidth]{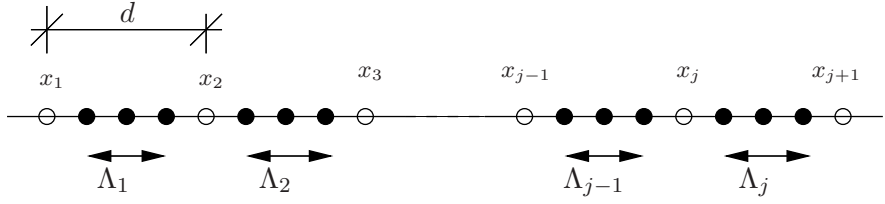}
\caption{Points $x_1,\dots , x_{j+1}$ and intervals $\L_1,\dots ,
\L_j$. } \label{out}
 \end{figure}

Finally, introduce the following hitting times:
$$
\tilde \tau_i = \inf \{t \geq 0 : \sigma_t(x_i)=1\}
\quad \mbox{and} \quad
\tau_i = \inf \{t \geq 0 : \sigma_t(y)=1 \mbox{ for any } y \in \Lambda_i \cup \{x_i\}\}  .
$$
Namely, $\tilde \tau_i$ is the first time that the vacancy at $x_i$
is removed, while $\tau_i$ is the first time that the box
$[x_i,x_{i+1}-1]$ appears totally filled. Note that, by
construction, it must be   $\tilde \tau_i <  \tau_i$ (trivially
$\tilde \tau_i \leq \tau_i$ and observe that when the zero is
removed from $x_i$ for the first time there must be a zero on
$x_i+1$).

\begin{Lemma} \label{chepalle}
For any positive integer $j \geq 3$, there exists a positive constant $c(j)$ such that, almost surely
$$
\liminf_{q \downarrow 0} \bbP_\sigma(\tilde \tau_2 < \tau_2 < \tilde \tau_3 < \tau_3 < \dots < \tilde \tau_{j-1} < \tau_{j-1} < \min(\tilde \tau_1, \tau_{j}) ) \geq c(j) .
$$
\end{Lemma}

\begin{proof}[Proof of Lemma \ref{chepalle}]
   Recall that, for any interval $\Lambda$, $\bbP^{\Lambda}_\sigma$
denotes the law of the process on the interval $\Lambda$, with empty
boundary condition, starting from $\sigma$. Fix a parameter $\delta
\in (0,1)$ (that will be chosen later) and, for $i,k=1,\dots, j-1$,
let $E_i^k$ be the event that, during the time interval
$$I_k:=[\frac{\delta(k-1)}{j-2}t_n, \frac{\delta k}{j-2}t_n]\,,$$ there have
never been simultaneously $n$ empty sites in the interval
$\Lambda_i$:
$$
E_i^k = \Big\{
 \sum_{y \in
\Lambda_i}\{ 1 - \sigma_s(y)\} \leq n-1 \qquad\forall s \in I_k \Big
\} .
$$
Also, set $A, B$ for the events that, in the time interval
$[0,\delta t_n]$, there have never been simultaneously $n$ empty
sites in the boxes $\L_1, \Lambda_j$ respectively:
\begin{align*}
& A=\Big \{  \sum_{y \in \Lambda_1} \{1 - \sigma_s(y)\} \leq n-1
\qquad \forall s \in [0, \delta t_n] \Big\} \,,\\
& B=\Big \{  \sum_{y \in \Lambda_j} \{1 - \sigma_s(y)\} \leq n-1
\qquad \forall s \in [0, \delta t_n] \Big\} .
\end{align*}
Finally, let $G$ be the event of the Lemma,  $F_i^k=\{ \tau_i \in
I_k
\}$ and
 $$ G^k_r= \left( \cap_{ k+2\leq \ell \leq  r } E_{\ell}^k \right) \cap F_{k+1}^k\,,
\qquad k=1,\dots,j-2, \qquad r\geq k+2. $$
By words, the event $G^k_r$ can be described
as follows: during the time interval $I_k$
the box $[x_{k+1},x_{k+2})= \{x_{k+1} \} \cup \L_{k+1} $ appears
totally filled for the first time, while in the above time interval
$I_k$  there are always less than $n$ simultaneous  empty sites in
all other boxes $\L_\ell=(x_\ell, x_{\ell+1})$ with $k+2 \leq \ell
\leq r$.


Figure \ref{fig:chepalle} illustrates the event $G$ of the lemma (a
special realization of it) and the evolution of the configuration
(with positive probability) from time $0$ till time $\delta t_n$.
This may help the reader to follow the proof.

\begin{figure}[ht]
\psfrag{t}{\!\!\!\!\!\! time $t$}
\psfrag{t0}{\footnotesize$0$}
\psfrag{t1}{\footnotesize $\frac{\delta t_n}{j-2}$}
\psfrag{t2}{\footnotesize $\frac{2\delta t_n}{j-2}$}
\psfrag{t3}{\footnotesize $\frac{3\delta t_n}{j-2}$}
\psfrag{t4}{\footnotesize $\delta t_n$}
\psfrag{tbar}{\footnotesize $\bar t=\frac{\delta (j-3)t_n}{j-2}$}
 \psfrag{x1}{\footnotesize $x_1$}
 \psfrag{x2}{\footnotesize $x_2$}
 \psfrag{x3}{\footnotesize$x_3$}
 \psfrag{x4}{\footnotesize $x_4$}
 \psfrag{x5}{\footnotesize $x_5$}
 \psfrag{xj-1}{\footnotesize$x_{j-1}$}
 \psfrag{xj}{\footnotesize$x_j$}
 \psfrag{xj+1}{\footnotesize$x_{j+1}$}
\psfrag{l1}{$\!\!\! \Lambda_1$}
\psfrag{l2}{$\!\!\!\!\Lambda_2$}
\psfrag{l3}{$\!\!\! \Lambda_3$}
\psfrag{l4}{$\!\!\!\!\Lambda_4$}
\psfrag{lj-1}{$\!\!\Lambda_{j-1}$}
\psfrag{lj}{$\!\!\Lambda_{j}$}
 \includegraphics[width=.80\columnwidth]{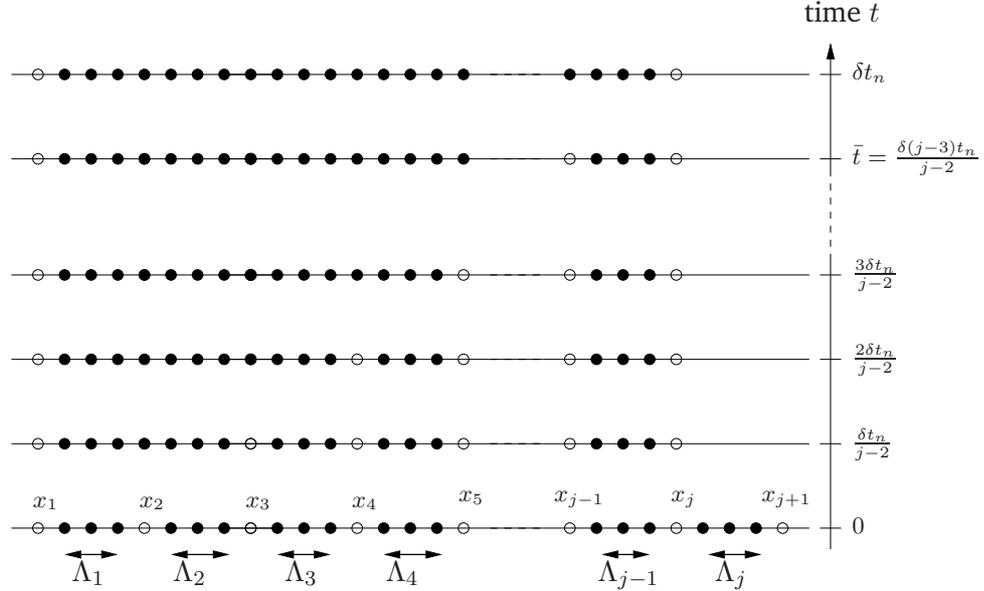}
 \caption{The event $G$ is implied by the evolution illustrated on
the picture.
Site $x_{j}$ remains empty (up to time $\delta t_n$) and acts as a
boundary condition of $[x_1,x_j-1]$ during this time interval. $\bar
\sigma$ is  the configuration  at time $\bar t=\delta
t_n(j-3)/(j-2)$.} \label{fig:chepalle}
 \end{figure}


 We now consider the event  $\cap_{k=1}^{j-2} G^k_{j-1} $. Due to Remark
 \ref{rem:comb},  for  $1\leq k \leq j-2$   the empty site $x_{k+1}$  remains empty during  the time intervals $I_1 \cup I_2 \cup
 \dots\cup
 I_{k-1}$   and is filled  together with the
 whole box $\L_{k+1}$ during the time interval $I_{k}$. In
 particular, the event $\cap_{k=1}^{j-2} G^k_{j-1} $ implies that
 $$\tilde \tau_2 < \tau_2 < \tilde \tau_3 < \tau_3 < \dots < \tilde \tau_{j-1} <
 \tau_{j-1}\,. $$
Always by Remark  \ref{rem:comb} the events $A,B$ imply that the
empty sites $x_1,x_j$ remain empty up to time $\d t_n$. As a
consequence we conclude that
$$ \left(\cap_{k=1}^{j-2} G^k_{j-1} \right) \cap A\cap B \subset G\,.
$$
so that \begin{equation}\label{passo1} \bbP_\sigma(G) \geq \bbP_\s
\bigl(A \tc  (\cap_{k=1}^{j-2} G^k_{j-1}) \cap   B\bigr)
\bbP_\sigma\bigl( \cap_{k=1}^{j-2} G^k_{j-1} \tc B\bigr)
\bbP_\sigma( B)\, .
\end{equation}
By Lemma \ref{zeri} below, it holds
\begin{equation}\label{passo2}
\bbP_\s  \bigl(A \tc  (\cap_{k=1}^{j-2} G^k_{j-1}) \cap   B\bigr) ,
\bbP_\sigma( B) \geq 1 - d \delta .
\end{equation}
On the other hand, as already explained above, the event $B$ guarantees that $\sigma_s(x_j)=0$ for all $s \in [0,\delta t_n]$,
so that, up to time $\delta t_n$, the process in infinite volume restricted to $[x_1,x_{j}-1]$
 coincides with the process in the finite volume $[x_1,x_{j}-1]$ (with empty boundary condition at site $x_j$).
Hence,
\begin{equation}\label{passo3}
\bbP_\sigma( \cap_{k=1}^{j-2} G^k_{j-1} \tc B) =
\bbP_\sigma^{[x_1,x_{j}-1]}( \cap_{k=1}^{j-2} G^k_{j-1} ) .
\end{equation}
Our aim is to bound the latter inductively. 
To that purpose  set $$X=\bbP_{\sigma_{0\mathds{1}}}^{[1,d]}\left(\t
\leq \d t_n /(j-2)\right )\,,$$ \ie  the probability that, starting
from $\sigma_{0\mathds{1}}$, within time $\delta t_n/(j-2)$ the box
$[1,d]$ appears completely filled.


Let also $\bar \sigma$ be the configuration obtained from $\sigma$
by removing the empty sites $x_2$, $x_3,\dots,x_{j-2}$ ({\it i.e.}\
$\bar \sigma(x)=\sigma(x)$ for all $x \neq x_2,x_3,\dots,x_{j-2}$
and  $\bar \sigma(x_i)=1$ for $i=2,3,\dots,j-2$, see Figure
\ref{fig:chepalle}). Set for simplicity $\bar t = \frac{\delta
(j-3)t_n}{j-2}$. Conditioning on the $\sigma$-algebra $\cF_{\bar t}$
generated by the Poisson processes and coin tosses up to time $\bar
t$, and using the Markov Property, we conclude that
\begin{align}
\bbP_\sigma^{[x_1,x_{j}-1]}( \cap_{k=1}^{j-2} G^k_{j-1} ) & \geq
X  \bbP_\sigma^{[x_1,x_{j}-1]}( (\cap_{k=1}^{j-3} G^k_{j-1}) \cap \{\sigma_{\bar t} = \bar \sigma \} ) \nonumber \\
& \geq X  \left[ \bbP_\sigma^{[x_1,x_{j}-1]}( \cap_{k=1}^{j-3}
G^k_{j-1}) - \bbP_\sigma^{[x_1,x_{j}-1]}((\cap_{k=1}^{j-3}
G^k_{j-1}) \cap \{ \sigma_{\bar t} \neq \bar \sigma \}) \right] .\label{passo4}
\end{align}
We deal with the two terms of the latter separately. Set $\bar B$ for the event that, up to time $\bar t$,
there have never been simultaneously $n$ empty sites in the interval $\Lambda_{j-1}$:
$$
\bar B= \Big \{  \forall s \in [0, \bar t\,], \sum_{y \in
\Lambda_{j-1}} \{1 - \sigma_s(y)\} \leq n-1\Big \}\,  .
$$
Then, we observe that  $\cap_{k=1}^{j-3} G^k_{j-1} =
 \cap_{k=1}^{j-3}
G^k_{j-1} \cap \bar B$. Hence, again using Lemma \ref{zeri} below
and the fact that the event $\bar B$ guarantees that
$\sigma_s(x_{j-1})=0$ for all $s \in [0, \bar t]$, we have
\begin{align}
\bbP_\sigma^{[x_1,x_{j}-1]}( \cap_{k=1}^{j-3} G^k_{j-1}) & =
\bbP_\sigma^{[x_1,x_{j}-1]}( \cap_{k=1}^{j-3} G^k_{j-1}  \tc \bar B)
\bbP_\sigma^{[x_1,x_{j}-1]}( \bar B) \nonumber  \\
& \geq \bbP_\sigma^{[x_1,x_{j-1}-1]}( \cap_{k=1}^{j-3} G^k_{j-2} )
(1-d \delta ) .\label{passo5}
\end{align}
Due to \eqref{passo4} and \eqref{passo5},  from Claim \ref{inferno1}
and Claim  \ref{inferno2} below, we conclude that, for $\delta$ and
$q$ small enough,
$$
\bbP_\sigma^{[x_1,x_{j}-1]}( \cap_{k=1}^{j-2} G^k_{j-1} ) \geq c_1
\bbP_\sigma^{[x_1,x_{j-1}-1]}( \cap_{k=1}^{j-3} G^k_{j-2} ) -c_2 q
$$
for some constants $c_1,c_2>0$ independent on $q$. A simple
iteration (adapting Claim \ref{inferno2})  allows us to end up with
\begin{equation}\label{montevettore}
\bbP_\sigma^{[x_1,x_{j}-1]}( \cap_{k=1}^{j-2} G^k_{j-1} ) \geq c_1'
\bbP_\sigma^{[x_1,x_3-1]}( F^1_2) - c_2 q  =  c_1' X -c_2 q\,.
\end{equation}
for some constant $c'$ depending on $j$ but independent from  $q$.
The result of Lemma \ref{chepalle} follows then from \eqref{passo1},
\eqref{passo2}, \eqref{passo3} and \eqref{montevettore}.

\begin{claim}\label{inferno1}
There exists $q_o, \delta_o>0$ and a constant $c=c(d,j)>0$ (independent on $q$) such that
$$
X:=\bbP_{\sigma_{0\mathds{1}}}^{[1,d]}(\tau \leq \delta t_n/(j-2) )
\geq c \qquad \forall \delta \in (0,\delta_o),\quad \forall q \in
(0,q_o) \,.
$$
\end{claim}

\begin{claim}\label{inferno2}
There exists a constant $c=c(j)$ (independent on $q$) such that
$$
\bbP_\sigma^{[x_1,x_{j}-1]}((\cap_{k=1}^{j-3} G^k_{j-1}) \cap  \{
\sigma_{\bar t} \neq \bar \sigma \}) \leq c q .
$$
\end{claim}

\begin{proof}[Proof of Claim \ref{inferno1}]
 Set $M=\delta t_n /[(j-2) T_n]$
where $T_n:=(1/q)^{(n-1)(1+3\e)}$ is defined in \cite[Section 3.2]{FMRT-cmp} with some fixed (small) parameter $\e>0$. Hence, one has to study
$Y_M := \bbP_{\sigma_{0\mathds{1}}}^{[1,d]}(\tau > \delta t_n/(j-2)) =
\bbP_{\sigma_{0\mathds{1}}}^{[1,d]}(\tau > M T_n)$. To bound $Y_M$ we use an induction procedure.
Namely, by the Strong Markov Property, we have
\begin{align*}
Y_M
&=
\bbP_{\sigma_{0\mathds{1}}}^{[1,d]}(\tau > (M-1) T_n + T_n)
=
\bbE_{\sigma_{0\mathds{1}}}^{[1,d]} \left( \mathds{1}_{\tau > (M-1)T_n} \bbP_{\sigma_{(M-1)T_n}}^{[1,d]}(\tau > T_n) \right) \\
& =
\bbE_{\sigma_{0\mathds{1}}}^{[1,d]} \left( \mathds{1}_{\tau > (M-1)T_n} \bbP_{\sigma_{0\mathds{1}}}^{[1,d]}(\tau > T_n) \right) \\
& \quad +
\bbE_{\sigma_{0\mathds{1}}}^{[1,d]} \left( \mathds{1}_{\tau > (M-1)T_n} \bbP_{\sigma_{(M-1)T_n}}^{[1,d]}(\tau > T_n) \mathds{1}_{\sigma_{(M-1)T_n} \neq \sigma_{0\mathds{1}}}\right) \\
& \leq
Y_{M-1} Y_1 + \bbP_{\sigma_{0\mathds{1}}}^{[1,d]}(\{\tau > (M-1) T_n\} \cap \{\sigma_{(M-1)T_n} \neq \sigma_{0\mathds{1}} \}) .
\end{align*}
 We observe that the event $\{\t > (M-1)T_n\}$  together with
$\{ \sigma_{(M-1)T_n} \neq \sigma_{0\mathds{1}}\}$ guarantees that
the configuration $\sigma_{(M-1)T_n}$ has an empty site in $[2,d]$,
which was not present in $\sigma_{0\mathds{1}}$. Indeed, if that was
false then it should be $\sigma_{(M-1)T_n}=\sigma_{0\mathds{1}}$ or
$\sigma_{(M-1)T_n}=\sigma_{\mathds{1}}$, thus leading to a
contradiction.   Hence, by Lemma \ref{zeri} below, we get
$$
\bbP_{\sigma_{0\mathds{1}}}^{[1,d]}(\{ \tau > (M-1) T_n \} \cap \{ \sigma_{(M-1)T_n} \neq \sigma_{0\mathds{1}} \})
 \leq q .
$$
In turn, $Y_M \leq Y_{M-1}Y_1 + q$. Set $X_k=Y_k-\frac{q}{1-Y_1}$ so
that $X_M \leq X_{M-1} Y_1$ which, after iteration, and using that
$-q/(1-Y_1) \leq 0$, leads to $Y_M \leq Y_1^M + \frac{q}{1-Y_1}$.

To end the proof we need to examine the term $Y_1=\bbP_{\sigma_{0\mathds{1}}}^{[1,d]}(\tau > T_n)$. To that aim,
using $\tilde \tau$ defined above,  we have the following decomposition
$$
Y_1=\bbP_{\sigma_{0\mathds{1}}}^{[1,d]}(\tilde \tau > T_n)
+ \bbP_{\sigma_{0\mathds{1}}}^{[1,d]}( \tau > T_n> \tilde \tau) .
$$
As above, the event $\{\tau > T_n> \tilde \tau\}$ implies that, at time $T_n$ there is an empty site in $[2,d]$, which was not present in $\sigma_{0\mathds{1}}$. Hence, by Lemma \ref{zeri} below
$\bbP_{\sigma_{0\mathds{1}}}^{[1,d]}( \tau > T_n> \tilde \tau) \leq q$. Applying Lemma \ref{rates} to the first term in the right hand side of the latter, it follows that
$Y_1 \leq \exp\{-cT_n/t_n\} + q$ for some constant $c$ that does not depend on $q$. Therefore,
as soon as $\delta$ is small enough, expanding in the limit $q \to 0$, we get
$$
Y_M \leq \left( \exp\{-cT_n/t_n\} + q \right)^M +
\frac{q}{1-\exp\{-cT_n/t_n\} - q} = \exp\{-c\delta/(j-2)\} + o(1)
$$
where $o(1)$ goes to zero as $q$ goes to zero.

All the previous computations lead to
$$
X =1-Y_M\geq 1-\exp\{-c\delta/(j-2) + o(1)\, .
$$
This ends the proof of the claim.
\end{proof}

\begin{proof}[Proof of Claim \ref{inferno2}]
First we observe that $\sigma_{\bar t} \neq \bar \sigma$ implies
either that $(a)$ there exists an empty site, at time $\bar t$ that
was not present at time $0$ ({\it i.e.} there exists an empty site
in $[x_1+1,x_{j}-1] \setminus \{x_2,x_3,\dots,x_{j-1}\}$), or $(b)$
at least one of the site $x_k$, $k \in \{2,3,\dots,j-2\}$ is empty.
Thanks to Lemma \ref{zeri} below, case $(a)$ has probability less or
equal to $q$ and we can focus on case $(b)$. Assume for simplicity
that $x_2$ is an empty site at time $\bar t$ (the other cases can be
treated analogously). We follows the lines of \cite[Lemma
4.2]{FMRT-cmp}, appealing to the graphical construction of Section
\ref{graphical}.
Given $m \geq 0$,  we write $\cA_{m}$ for the event that the last
legal ring at $x_2$ before time $\bar t$ (which is well defined
because $x_2$ has been filled during the time interval $[0,\delta
t_n/(j-2)]$) occurs at time $t_{y_1,m}$. Recall that (i) at the time
$t_{y_1,m}$ the current configuration resets its value at $x_2$ to
the value of an independent Bernoulli$(1-q)$ random variable
$s_{x_2,m}$ and (ii) that $\cA_{m}$ depends only on the Poisson
processes associated to sites $x\ge x_2$ and on the Bernoulli
variables associated to sites $x>x_2$. Hence we conclude that (we
drop the superscript $[x_1,x_{j}-1]$)
\begin{align*}
\bbP_\sigma((\cap_{k=1}^{j-3} G^k_{j-1} ) \cap \{\sigma_{\bar
t}(x_2) =0 \} ) & = \bbP_\sigma
\Big(\cup_{m=1}^{\infty} \bigl(\cA_{m} \cap \{s_{x_2,m}=0\}\bigr)\cap\bigl(\cap_{k=1}^{j-3} G^k_{j-1} \bigr) \Big)\\
& \leq  \sum_{m=1}^\infty \bbP_\sigma(s_{x_2,m}=0)\bbP_\sigma (\cA_m
)
 \leq q .
\end{align*}
The claim follows.
\end{proof}

The proof of Lemma \ref{chepalle} is complete.
\end{proof}
The proof of Theorem \ref{plateau} is complete.
\end{proof}

Below we recall some useful technical facts borrowed from \cite{FMRT-cmp}.

Given a configuration $\sigma$, let $\cZ(\sigma)=\{x \in \bbZ : \sigma(x)=0\}$ be the set of all empty sites of $\sigma$.

\begin{Lemma}[Lemma 4.2 of \cite{FMRT-cmp}] \label{zeri}
Fix $\sigma\in\O$, $t\geq 0$ and $k\in\bbN$. Let
$V=[0,a]\sset \bbZ$ and let
$\{y_1,\dots, y_k\}\sset V\setminus \cZ(\s)$. Let finally $\cF$ be the
$\s$-algebra generated by the Poisson processes and coin tosses in
$\bbZ \setminus V$. Then
\begin{equation} \label{eq2}
\bbP^\L_{\sigma}\bigl(\bigl\{y_1,\dots, y_k\}\sset \cZ(\s_s)\tc \cF\bigr)\leq q^k, \qquad \forall s>0.
 \end{equation}
Moreover
\begin{equation} \label{eq2.1}
\bbP_{\sigma}\bigl(\exists\, s\le t:\ \bigl\{y_1,\dots, y_k\}\sset
\cZ(\s_s)\tc \cF\bigr)\leq at q^k\,.
 \end{equation}
\end{Lemma}

\begin{Lemma}\label{rates}
Let $T_n=(1/q)^{(n-1)(1+3\e)}$ for some fixed parameter $\e \in (0,1)$.
Let $\sigma_{0\mathds{1}}$ be the configuration, on $[1,d]$, with only one empty site at $1$,
and $\tilde \tau:= \inf \{ s: \sigma_s(1)=1\}$.Then,
$$
\bbP_{\sigma_{0\mathds{1}}}^{[1,d]}(\tilde \tau > T_n) \leq \exp \{ -cT_n/t_n\}
$$
for some constant $c=c(d,\e)$ that does not depend on $q$.
\end{Lemma}

\begin{proof}[Proof of Lemma \ref{rates}]
The result of Lemma \ref{rates} follows from \cite[Lemma 3.4]{FMRT-cmp} together with the definition of $\l_n(d)$
given in \cite[Equation (3.4)]{FMRT-cmp}.
\end{proof}

\begin{remark}
Note that, in \cite[Lemma 4.2]{FMRT-cmp}, the result holds for configurations living in $\Omega_{\bbZ_+}$. However, the proof
can easily be adapted to $\Omega=\Omega_\bbZ$ as stated in Lemma \ref{zeri}.
\end{remark}


\subsection*{Acknowledgements}
We  thank   the Laboratoire de Probabilit\'{e}s et Mod\`{e}les
Al\'{e}atoires, the University Paris VII and the Department of
Mathematics of the University of Roma Tre for the  support and the
kind hospitality.
C. Toninelli acknowledges the partial support of the
French Ministry of Education
through the ANR BLAN07-2184264 grant.

\bibliographystyle{amsalpha}
\bibliography{East}

\end{document}